\documentclass{amsart}
\usepackage{amsthm}
\usepackage{amsmath}
\usepackage{amstext}
\usepackage{amsfonts}
\usepackage{amssymb}
\usepackage{latexsym}
\usepackage{amscd}
\usepackage{xypic}
\theoremstyle{plain}
\newtheorem{theorem}{Theorem}[section]

\newtheorem{proposition}[theorem]{Proposition}
\newtheorem{corollary}[theorem]{Corollary}

\theoremstyle{definition}

\newtheorem{example}[theorem]{Example}
\newtheorem{remark}[theorem]{Remark}

\numberwithin{equation}{theorem}
\newtheorem{claim}[theorem]{Claim}
\begin{document}

\title[Vector fields and moduli of canonically pol. surfaces in char.  $P>0$.]{Vector fields and moduli of canonically polarized surfaces in positive characteristic.}
\author{Nikolaos Tziolas}
\address{Department of Mathematics, University of Cyprus, P.O. Box 20537, Nicosia, 1678, Cyprus}
\email{tziolas@ucy.ac.cy}




\begin{abstract}
This paper investigates the geometry of smooth canonically polarized surfaces defined over a field of positive characteristic which have a nontrivial global vector field, and the implications that the existence of such surfaces has in the moduli problem of canonically polarized surfaces.

In particular, an explicit real valued function $f(x)$ is obtained such that if $X$ is a smooth canonically polarized surface defined over an algebraically closed field of characteristic $p>0$ such that $K_X^2 <f(p)$, then $X$ is unirational and the order of its algebraic fundamental group is at most two. As a consequence of this result, large classes of canonically polarized surfaces are identified whose moduli stack is Deligne-Mumford, a property that does not hold in general in positive characteristic.
\end{abstract}

\maketitle

\section{Introduction}

The objective of this paper is to investigate the geometry of smooth canonically polarized surfaces with nontrivial global vector fields and to use the results of this investigation in order to study the moduli stack of canonically polarized surfaces in positive characteristic. An investigation with these objectives was initiated in~\cite{Tz17} where the case of surfaces $X$ with $K_X^2\leq 2$ has been studied.

A normal projective surface $X$ defined over an algebraically closed field is called canonically polarized if and only if $K_X$ is ample. Canonically polarized surfaces with canonical singularities appear as the canonical models of smooth surfaces of general type and they play a fundamental role in the classification problem of surfaces of general type. In fact, early on in the theory of moduli of surfaces of general type, it was realized that the moduli functor of surfaces of general type is not well behaved and that the correct objects to parametrize are not the surfaces of general type but instead their canonical models~\cite{Ko10}. Canonically polarized surfaces are the two dimensional analogs of stable curves.

The property that a smooth canonically polarized surface $X$ has a nontrivial global vector field is equivalent to the property that its automorphism scheme $\mathrm{Aut}(X)$ is not smooth. The reason is that the space of global vector fields of $X$ is canonically isomorphic to $\mathrm{Hom}(\Omega_X,\mathcal{O}_X)$, the tangent space at the identity of $\mathrm{Aut}(X)$. Moreover, it is well known that if $X$ is canonically polarized then $\mathrm{Aut}(X)$ is a zero dimensional scheme of finite type over the base field. Therefore the existence of nontrivial global vector fields on $X$ is equivalent to the non smoothness of $\mathrm{Aut}(X)$. Considering that $\mathrm{Aut}(X)$ is a group scheme and every group scheme in characteristic zero is smooth, non smoothness of $\mathrm{Aut}(X)$ can happen only in positive characteristic. Consequently a smooth canonically polarized surface can have non trivial global vector fields only when it is defined over a field of positive characteristic. Examples of such surfaces have been found by H. Kurke~\cite{Ku81}, W. Lang~\cite{La83} and N. I. Shepherd-Barron~\cite{SB96}.

The existence of nontrivial global vector fields on canonically polarized surfaces is intimately related to fundamental properties of their moduli functor and in particular their moduli stack. It is well known that in characteristic zero the moduli functor of canonically polarized surfaces with fixed Hilbert polynomial  has a separated coarse moduli space  which is of finite type over the base field $k$. Moreover, the moduli stack of stable surfaces is a separated, Deligne-Mumford stack of finite type~\cite{KSB88}~\cite{Ko97}. In positive characteristic, a coarse moduli space still exists~\cite{Ko97} but the moduli stack is not always Deligne-Mumford. The reason for this failure is the existence of smooth canonically polarized surfaces with non smooth automorphism scheme, or equivalently with nontrivial global vector fields~\cite[Theorem 4.1]{DM69}. In some sense then the existence of nontrivial global vector fields on canonically polarized surfaces is the obstruction for the moduli stack to be Deligne-Mumford.

The failure of the moduli stack to be Deligne-Mumford is rather unpleasant for the following reason. The ideal outcome of the classification problem is the existence of a universal family. A family from which all other families of objects in the moduli problem should be obtained after a base change. However, due to the presence of automorphisms of the parametrized objects, a universal family rarely exists. Then one relaxes the requirements of a universal family and studies the corresponding moduli stack. If the stack is Deligne-Mumford, then there exists a family $\mathcal{X} \rightarrow S$ such that for any variety $X$ in the moduli problem, there exists finitely many $s\in S$ such that $\mathcal{X}_s\cong X$, up to \'etale base change any other family is obtained from it by base change and that for any closed point $s \in S$, the completion $\hat{\mathcal{O}}_{S,s}$ pro-represents the local deformation functor $Def(X_s)$. In other words this family is universal in the \'etale toplogy and provides  a connection between the local moduli functor and the global one. 

This  investigation has two main objectives. 

The first objective is to find numerical conditions, preferably deformation invariant, which imply that the moduli stack of smooth canonically polarized surfaces is Deligne-Mumford. According to~\cite[Theorem 3.1]{Tz17} such conditions exist. However their existence is due to purely theoretical reasons and no explicit conditions were obtained so far.

The second  objective is  to describe the geometry of canonically polarized surfaces which have nontrivial global vector fields and consequently their moduli stack is not Deligne-Mumford. The hope is to obtain a good insight in the geometry of such surfaces that will allow the modification of the moduli problem in order to get a better moduli theory for these surfaces. 

From the existing examples of canonically polarized surfaces with nontrivial global vector fields and the case of surfaces with $K^2\leq 2$, one gets the feeling that surfaces with nontrivial global vector fields tend to be uniruled and simply connected~\cite{Tz17}. However non uniruled examples exist in characteristic 2~\cite{SB96}, but it is unknown if non uniruled examples exist in higher characteristics. 

The main result of this paper is the following. 


\begin{theorem}\label{intro-the-1}
Let $X$ be a smooth canonically polarized surface defined over an algebraically closed field of characteristic $p>0$. Suppose that $\mathrm{Aut}(X)$ is not smooth, or equivalently that there exists a nontrivial global vector field on $X$.  Suppose that 
\[
K_X^2 <\mathrm{min}\left\{\frac{1}{2}\sqrt[3]{\frac{p-3}{2}},\frac{p-3}{3960}\right\}. 
\]
Then $X$ is unirational, $b_1(X)=0$ and $|\pi_1(X)|\leq 2$.
\end{theorem}

Notice that the conditions on $b_1$ and the fundamental group are deformation invariant and hence are good conditions for the moduli problem. 

If the automorphism scheme $\mathrm{Aut}(X)$ of $X$ is not smooth then $\mathrm{Aut}(X)$ contains a subgroup scheme isomorphic to either $\alpha_p$ or $\mu_p$. This is equivalent to say that if $X$ has a nontrivial global vector field then $X$ has a nontrivial global vector field $D$ such that $D^p=0$ or $D^p=D$~\cite{Tz17b},~\cite{RS76}. If $\mu_p$ is a subgroup scheme of $\mathrm{Aut}(X)$, then finer restrictions can be imposed on $K_X^2$ which imply the unirationality of $X$.

\begin{theorem}\label{intro-the-2}
Let $X$ be a smooth canonically polarized surface defined over an algebraically closed field of characteristic $p>0$. Suppose that  $\mu_p\subset \mathrm{Aut}(X)$, or equivalently that $X$ has a nontrivial vector field of multiplicative type. Suppose that $K_X^2<(p-3)/156$. Then $X$ is unirational, $b_1(X)=0$ and $|\pi_1(X)|\leq 2$.
\end{theorem}

\begin{corollary}\label{intro-cor-1}
Let $X$ be a smooth canonically polarized surface defined over an algebraically closed field of characteristic $p>0$ such that 
\[
K_X^2 <\mathrm{min}\left\{\frac{1}{2}\sqrt[3]{\frac{p-3}{2}},\frac{p-3}{3960}\right\}
\]
and such that either $b_1(X)\not= 0$ or that $|\pi_1(X)|>2$. Then $\mathrm{Aut}(X)$ is smooth and therefore there do not exist nontrivial global vector fields on $X$.
\end{corollary}

The previous results have immediate applications to the structure of the moduli stack of canonically polarized surfaces.

\begin{theorem}\label{intro-the-3}
Let $k$ be a field of characteristic $p>0$ and $\lambda <\mathrm{min}\left\{\frac{1}{2}\sqrt[3]{\frac{p-3}{2}},\frac{p-3}{3960}\right\}$ be a positive integer. Let $\mathcal{M}_{\lambda, 3}$ be the moduli stack of smooth canonically polarized surfaces with $K^2= \lambda$ and fundamental group of order at least 3. Then  $\mathcal{M}_{\lambda, 3}$ is Deligne-Mumford.
\end{theorem}
Taking into consideration the breadth of the possible values of the fundamental group of canonically polarized surfaces (it can be finite or infinite)~\cite{BCP11}, one sees that the previous results applies to a very large class of canonically polarized surfaces.

There are three comments that I would like to make regarding the statement of Theorems~\ref{intro-the-1},~\ref{intro-the-2}.

The bounds on $K^2$ are not optimal. In particular, if $K_X^2=1$, then Theorem~\ref{intro-the-1} says that $X$ is unirational if $p>3963$. But, according to~\cite[Theorem 1.1]{Tz17}, if $K_X^2=1$ then $X$ is unirational for all $p$ except possibly for $p= 3,5,7$. However, I believe that the strength of Theorem~\ref{intro-the-1} lies in its generality and not the optimality of the bounds presented. The results apply to every canonically polarized surface and not to a specific class of them. In individual cases, like the cases when $K_X^2 \leq 2$ which have been treated in~\cite{Tz17} finer results might be obtained by exploiting known results about the geometry of the surfaces in question. I believe that a more refined version of the method used to prove the theorems should provide a better bound for $K^2$. 

A desired result would be to obtain a bound for $K^2$ in the form $K^2<f(p)$, where $f(p)$ is a function of $p$, which implies the smoothness of 
$\mathrm{Aut}(X)$. Such a result will make it possible to obtain a theorem like Theorem~\ref{intro-the-3} for canonically polarized surfaces whose fundamental group has order at most 2 as well. However, the bounds for $p$ are most likely going to be larger than those in Theorems~\ref{intro-the-1},~\ref{intro-the-2} making such a result weaker, since it would cover less cases, compared to Theorems~\ref{intro-the-1},~\ref{intro-the-2} for surfaces whose fundamental group has order at least 3. I believe that a method based on the methods used in this paper should provide such a bound. However, at the moment I am unable to do so.

The reason that in Theorem~\ref{intro-the-2} I was able to obtain a much better bound for $K_X^2$ in the case when $X$ has a vector field of multiplicative type, or equivalently when $\mu_p$ is a subgroup scheme of $\mathrm{Aut}(X)$, is that $\mu_p$ is a diagonalizable group scheme while $\alpha_p$ is not.

This paper is organized as follows.

Section~\ref{sec-1} contains some simple  results that are necessary for the proof of the main theorems.

In Section~\ref{sec-2} the general method and strategy for the proof of Theorems~\ref{intro-the-1},~\ref{intro-the-2} are explicitly described. 

In Section~\ref{sec-3} the case of canonically polarized surfaces which have a nontrivial global vector field $D$ which has only isolated singularities is studied.

In Section~\ref{sec-4} the case of canonically polarized surfaces which have a nontrivial global vector field $D$ whose singularities have a divisorial part is studied.

Finally, the statements  of Theorems~\ref{intro-the-1},~\ref{intro-the-2}  is the combination of the statements of Theorems~\ref{sec3-th-1},~\ref{sec4-theorem-1}.

\section{Notation-Terminology}
Let $X$ be a scheme of finite type over a field $k$ of characteristic $p>0$.

Let $\mathcal{F}$ be a coherent sheaf on $X$. By $\mathcal{F}^{[n]}$ we denote the double dual $(\mathcal{F}^{\otimes n})^{\ast\ast}$. 

$X$ is called a smooth canonically polarized variety if and only if $X$ is a smooth  and $\omega_X$ is ample.

$Der_k(X)$ denotes the space of global $k$-derivations of $X$ (or equivalently of global vector fields). It is canonically identified with $\mathrm{Hom}_X(\Omega_X,\mathcal{O}_X)$.

Let $D$ be a nontrivial global vector field on $X$. $D$ is called $p$-closed if and only if $D^p=\lambda D$, for some $\lambda\in k$. $D$ is called of additive type if $D^p=0$ and of multiplicative type if $D^p=D$.
The fixed locus of $D$ is the closed subscheme of $X$ defined by the ideal sheaf $(D(\mathcal{O}_X))$. 
The divisorial part of the fixed locus of $D$ is called the divisorial part of $D$.  A point $P\in X$ is called an isolated singularity of $D$ if and only if the ideal of $\mathcal{O}_{X,P}$ generated by $D(\mathcal{O}_{X,P})$  has an associated prime of height $\geq 2$.

A prime divisor $Z$ of $X$ is called an integral divisor of $D$ if and only if locally there is a derivation $D^{\prime}$ of $X$ such that $D=fD^{\prime}$, $f \in K(X)$,  $D^{\prime}(I_Z)\subset I_Z$ and $D^{\prime}(\mathcal{O}_X) \not\subset I_Z$ ~\cite{RS76}. 

Let $X$ be a normal surface and $D$ a nontrivial global vector field on $X$ of either additive of multiplicative type. Then $D$ induces an $\alpha_p$ or $\mu_p$ action on $X$. Let $\pi \colon X\rightarrow Y$ be the quotient of $X$ by this action~\cite[Theorem 1, Page 104]{Mu70}. Let $C\subset X$ be a reduced and irreducible curve and $\tilde{C}=\pi(C)$. Suppose that $C$ is an integral curve of $D$. Then $\pi^{\ast}\tilde{C}=C$. Suppose that $C$ is not an integral curve of $D$. Then $\pi^{\ast}\tilde{C}=pC$~\cite{RS76}.

For any prime number $l\not= p$, the cohomology groups $H_{et}^i(X,\mathbb{Q}_l)$ are independent of $l$, they are finite dimensional of $\mathbb{Q}_l$ and are called the $l$-adic cohomology groups of $X$. The $i$-Betti number $b_i(X)$ of $X$ is defined to be the dimension of $H_{et}^i(X,\mathbb{Q}_l)$. It is well known that $b_i(X)=0$ for any $i>2n$, where $n=\dim X$~\cite[Chapter VI, Theorem 1.1]{Mi80}. 

$X$ is called  simply connected if $\pi_1(X)=\{1\}$, where $\pi_1(X)$ is the \'etale fundamental group of $X$.

Let $P\in X$ be a normal surface singularity and $f \colon Y \rightarrow X$ its minimal resolution. $P\in X$ is called DuVal (or canonical) if and only if $K_Y=f^{\ast}K_X$. By~\cite[Theorem 4.22]{KM98} canonical surface singularities are classified according to the Dynkin diagrams of their minimal resolution and they are called correspondingly of type $A_n$, $D_n$, $E_6$, $E_7$ and $E_8$. In characteristic zero these singularities are classified by explicit equations. However in positive characteristic I am not aware of a classification with respect to equations.

\section{Preparatory Results.}\label{sec-1}
Let $X$ be a smooth surface defined over an algebraically closed field $k$ of characteristic $p>0$. Let $D$ be a nontrivial vector field on $X$ (or equivalently a $k$-derivation of $\mathcal{O}_X$). 



The next proposition presents a method  to find integral curves of $D$.

\begin{proposition}[Proposition 6.10~\cite{Tz17}]\label{sec1-prop1}
Suppose that either $D^p=0$ or $D^p=D$ and let $\mathrm{Fix}(D)\subset X$ be the fixed subscheme of $X$ for the $\alpha_p$ or $\mu_p$ action on $X$ induced by $D$. Let $\pi \colon X \rightarrow Y$ be the quotient of $X$ by this action. Let $L$ be a rank one reflexive sheaf on $Y$ and $M=(\pi^{\ast}L)^{[1]}$. Then $D$ induces a $k$-linear map
\[
D^{\ast} \colon H^0(X,M) \rightarrow H^0(X,M)
\]
with the following properties:
\begin{enumerate}
\item $\mathrm{Ker}(D^{\ast})=H^0(Y,L)$ (considering $H^0(Y,L)$ as a subspace of $H^0(X,M)$ via the map $\pi^{\ast}$).
\item If $D^p=0$ then $D^{\ast}$ is nilpotent and if $D^p=D$ then $D^{\ast}$ is a diagonalizable map whose eigenvalues are in the set $\{0,1,\ldots,p-1\}$.
\item Let $s\in H^0(X,M)$ be an eigenvector of $D^{\ast}$. Then $D(I_{Z(s)})\subset I_{Z(s)}$, where  $Z(s)$ is the divisor of zeros of $s$. In particular, if $D^{\ast}(s)=\lambda s$, and $\lambda\not= 0$, then $(D(I_{Z(s)}))|_V=I_{Z(s)}|_V$, where $V=X-\pi^{-1}(W)$, $W\subset Y$ is the set of points that $L$ is not free.
 \end{enumerate}
\end{proposition}

The previous proposition shows that every eigenvector of $D^{\ast}$ corresponds to a curve $C\subset X$ such that $D(I_C)\subset I_C$ and therefore $D$ induces a vector field on $C$. However it is possible that $D(\mathcal{O}_X)\subset I_C$ and hence the induced vector field on $C$ is trivial. This implies that $C$ is contained in the divisorial part of $D$. This cannot happen of course if $D$ has only isolated singularities. 

Let $C=n_1C_1+\cdots +n_kC_k$ be a curve in $X$ and its decomposition into its prime components. Suppose that $D(I_C)\subset I_C$. In general $D$ does not induce vector fields on $C_i$, i.e, $D(I_{C_i})$ may not be contained in $C_i$. For example for any reduced and irreducible curve  $C$, $D$ fixes $pC$ but not necessarily $C$. The next proposition provides some conditions in order for $D$ to restrict to $D_i$.

\begin{proposition}\label{sec1-prop2}
Let $C \subset X$ be a curve such that $D(I_C)\subset I_C$, where $I_C \subset \mathcal{O}_X$ is the ideal sheaf of $C$ in $X$ and such that $D(\mathcal{O}_X) \not\subset I_C$, i.e, $C$ is not contained in the fixed locus of $D$. Let $C=n_1C_1+\cdots +n_kC_k$ be the decomposition of $C$ in its irreducible and reduced components. If $p$ does not divide $n_i$, for all $1\leq i \leq k$, then $D(I_{C_i})\subset I_{C_i}$, for all $1\leq i \leq k$. Therefore $D$ fixes the reduced part of every irreducible component of $C$ and hence induces a vector field on $C_i$, for all $1\leq i \leq k$.
\end{proposition}

\begin{proof}
Let $Q\in C_i$ be a closed point such that $Q\not\in C_j$, for any $j\not=i$, $1\leq i,j, \leq k$. Then locally around $Q$, $X=\mathrm{Spec}A$ and $D$ is a $k$-derivation of $A$. Since $X$ is smooth, $I_C=I_{C_i}=(t^{n_i})$, where $t \in A$ is a prime element. Then since $D(_C)\subset I_C$, it follows that $n_it^{n_i-1}Dt \in (t^{n_i})$ and hence there exists $a\in A$ such that $n_it^{n_i-1}Dt =at^{n_i}$. Now since $p$ does not divide $n_i$, $n_i \not=0$ in $k$ and hence it follows that $Dt \in (t)$. Hence $D(I_{C_i})|_U=I_{C_i}|_U$, where $U=X-Z$, $Z=C_i \cap (\cup_{j\not= i}C_j)$.

Let now $V=\mathrm{Spec}A \subset X$ be an affine open subset such that $I_{C_i}|_U \not= \mathcal{O}_U$. Let $a \in \cap_{j\not= i} I_{C_j}$ such that 
$a\not\in I_{C_i}$. Then $ I_C A_a=I_{C_i}A_a$ and therefore $D(I_{C_i}A_a)\subset I_{C_i}A_a$. Let $x\in I_{C_i}$. Then $D(x/1)=y/a^m$, for some 
$y\in I_{C_i}$ and $m\in \mathbb{N}$. Therefore $a^mDx \in I_{C_i}$ and since $a\not \in I_{C_i}$ it follows that $Dx \in I_{C_i}$. Hence
 $D(I_{C_i})\subset I_{C_i}$, for any $1\leq i \leq k$.
\end{proof}

\begin{corollary}\label{sec1-cor1}
With assumptions as in Proposition~\ref{sec1-prop2}. Suppose  in addition that $K_X$ is ample and that $K_X \cdot C < p$. Then $D(I_{C_i})\subset I_{C_i}$, for all $1\leq i \leq k$. Therefore $D$ fixes the reduced part of every irreducible component of $C$ and hence induces a vector field on $C_i$, for all $1\leq i \leq k$.
\end{corollary}

\begin{proof}
Since $K_X$ is assumed to be ample, the condition $K_X\cdot C <p$ immediately implies that $n_i <p$, for all $1\leq i \leq k$. Then the corollary follows directly from Proposition~\ref{sec1-prop2}.
\end{proof}

\begin{proposition}\label{sec1-prop3}
Suppose that $K_X$ is ample. Let $C\in |mK_X|$ be a curve such that $D(I_C) \subset I_C$. Let $C=n_1C_1+\cdots +n_kC_k$ its decomposition into its prime divisors. Suppose that $K_X^2< p/(m^2+3m)$. Then every point of intersection of $C_i$ and $C_j$, $i\not= j$, is a fixed point of $D$. 

\end{proposition}

\begin{proof}
The result is local at the points of intersection of $C_i$ and $C_j$. So let $P\in C_i \cap C_j$ be a point of intersection of $C_i$ and $C_j$. Let $U=\mathrm{Spec} A$ be an affine open subset of $X$ containing $P$ but no other point of $C_i \cap C_j$. Let $I$ and $J$ be the ideals of $C_i$ and $C_j$ respectively. Then $I+J=Q$, with $r(Q)=m_P$, the maximal ideal corresponding to the point of intersection $P$ of $C_i$ and $C_J$. By assumption, $D(I) \subset I$ and $D(J) \subset J$. Therefore $D(I+J)=D(I)+D(J)\subset I+J$. Hence $D(Q)\subset Q$. Ii will show that this implies that $D(m_P)\subset m_P$ and therefore $P$ is a fixed point of $D$.

In order to show that $D(m_P)\subset m_P$ I will first show that $C_i\cdot C_j <p$. Then if $I=(f)$ and $J=(g)$, $f,g \in A$, $\dim_k A/(f,g)< p$. Hence for any $a\in m_Q$, there exists $m<p$ such that $a^m \in Q=I+J$. Let $m_0<p$ be the smallest such $m$. Then $D(a^{m_0})=m_0a^{m_0-1}Da \in Q=I+J$. $Q$ is a primary ideal and $a^{m_0-1} \not\in Q$. Hence $(Da)^s \in Q \subset m_P$, for some $s \geq 0$. Hence $Da \in m_P$. Therefore $D(m_P)\subset m_P$, as claimed. 

It remains to show that $C_i \cdot C_j <p$. By definition, $mK_X \sim \sum_{s=1}^k n_s C_s$.  Let $1\leq i,j \leq k$. Then
\begin{gather}\label{sec1-prop3-eq1}
mK_X\cdot C_i=n_j C_i \cdot C_j +n_i C_i^2 +\sum_{s\not= i,j} n_s C_s \cdot C_i \geq n_j C_i \cdot C_j +n_iC_i^2.
\end{gather}
On the other hand, $mK_X^2=\sum_{s=1}^m n_s K_X\cdot C_s$ and since $K_X$ is ample, it follows that $K_X\cdot C_s>0$ for every $1\leq s \leq m$ and therefore $K_X\cdot C_s \leq n_s K_X \cdot C_s \leq mK_X^2$. Then from~\ref{sec1-prop3-eq1} it follows that
\begin{gather}\label{sec1-prop3-eq2}
C_i\cdot C_j \leq m^2K_X^2 -n_iC_i^2.
\end{gather}
Suppose that $C_i^2\geq 0$. Then from the above equation it follows that
\[
C_i \cdot C_j < m^2K_X^2<\frac{pm^2}{m^2+3m}<p.
\]
Suppose that $C_i^2 <0$. Then from the adjunction formula it follows that $C_i^2=2p_a(C_i)-2-K_X\cdot C_i\geq -2-K_X\cdot C_i$. Then from the equation~\ref{sec1-prop3-eq2} it follows that 
\begin{gather}\label{sec1-prop3-eq3}
C_i\cdot C_j \leq m^2K_X^2+2n_i+n_iK_X\cdot C_i.
\end{gather}
But it has been shown earlier that $n_iK_X\cdot C_i \leq mK_X^2$ and hence $n_i \leq mK_X^2$ and $K_X\cdot C_i <mK_X^2$. Hence 
\[
C_i \cdot C_j \leq (m^2+3m)K_X^2<\frac{p(m^2+3m)}{m^2+3m}<p,
\]
as claimed. This concludes the proof.

\end{proof}
The proof of the previous proposition shows also the following.
\begin{corollary}\label{sec1-cor0}
Let $C_1$, $C_2$ be two different irreducible and reduced curves on $X$ such that $D(I_{C_i})\subset I_{C_i}$, for $i=1,2$. Assume that $C_1 \cdot  C_2 <p$. Then every point of intersection of $C_1$ and $C_2$ is a fixed point of $D$.
\end{corollary}

Let $D$ be a vector field on a variety $X$. Then unlike the characteristic zero case~\cite{BW74}, $D$ does not fix the singular points of $X$. And even if it does fix a singular point, it may not fix its infinitely near points. 
\begin{example}
Let $X$ be given by $x^2-y^5=0$ in $\mathbb{A}^2_k$, where $k$ is any field of characteristic 2. Let $D=y\frac{\partial}{\partial x}$. Then it can be easily sen that $D$ is a vector field of additive type which fixes the singular point of $X$. Then $D$ lifts to a vector field $D^{\prime}$ of the blow up $X^{\prime} \rightarrow X$ of the singular point of $X$, but $D^{\prime}$ does not fix the singular point of $X^{\prime}$.
\end{example}

The next proposition shows that under certain conditions a vector field on a curve fixes the singular points of the curve.

\begin{proposition}\label{sec1-prop4}
Let $D$ be a nontrivial vector field of either additive or multiplicative type on a smooth surface $X$ defined over an algebraically closed field $k$ of characteristic $p>0$. Let $C \subset X$ be a reduced and irreducible curve such that $D(I_C)\subset I_C$, where $I_C$ is the ideal sheaf of $C$ in $X$. Suppose that $\mathrm{p}_a(C)<(p-1)/2$. Then $D$ fixes every singular and infinitely near singular point of $C$.
\end{proposition}

\begin{proof}
We may assume that $D(\mathcal{O}_X)\not\subset I_C$ and hence the restriction of $D$ on $C$ is not trivial (otherwise the result is obvious).

Let $\pi \colon X \rightarrow Y$ be the quotient of $X$ by the $\alpha_p$ or $\mu_p$ action on $X$ induced by $D$. Then $\pi$ is a purely inseparable morphism of degree $p$. Let $\tilde{C}=\pi(C)\subset Y$. Then $C=\pi^{\ast}\tilde{C}$ and $\pi_{\ast}C=p\tilde{C}$~\cite{RS76}. Let $P\in C$ be a singular point of $C$ and $Q=\pi(P)\in Y$. If $P$ is a fixed point of $D$ then there is nothing to prove. Suppose that $P$ is not a fixed point of $D$. Then $Q\in Y$ is a smooth point of $Y$~\cite{AA86}. Hence locally around $Q \in Y$, $X\rightarrow Y$ is an $\alpha_p$ or $\mu_p$ torsor and hence the same holds for $C\rightarrow \tilde{C}$. Consider cases with respect to whether $Q \in \tilde{C}$ is a singular or a smooth point of $C$.

\textbf{Case 1.} $Q\in \tilde{C}$ is singular. Then since $P\in X$ is not a fixed point of $D$, in suitable local analytic coordinates at $P$, $\mathcal{O}_X=k[[x,y]]$, $D=h(x,y) \partial/\partial x$ and $\mathcal{O}_Y=k[[x^p,y]]$~\cite[Theorem 1]{RS76}. Then $I_{\tilde{C}}=(f(x^p,y))$ and since it is assumed that $Q\in \tilde{C}$ is singular, $f(x^p,y)\in (x^p,y)^2$. Then $I_C=(f(x^p,y)) \subset k[[x,y]]$. Write $f(x^p,y)=\sum_i f_i(x^p)y^i$. Then either $m_P(f(x^p,y)) \geq p$ (considered in $k[[x,y]]$) or there exists an $m\geq 1$ such that $f_m(x^p)$ is a unit in $k[[x^p]]$. 

The first case is easily seen to be impossible since $C$ is assumed to have arithmetic genus less than $p$ and a curve of arithmetic genus less than $p$ cannot have a point of multiplicity bigger than $p$.

Suppose then that there exists an $m \geq 1$ such that $f_m(x^p)$ is a unit in $k[[x^p]]$. By using the Weierstrass preparation theorem in $k[[x^p,y]]$ it follows that 
\[
f(x^p,y)=u(x^p,y)[f_0(x^p)+f_1(x^p)y+\cdots + f_{m-1}(x^p)y^{m-1}+y^m],
\]
where $f_i(x^p)\in (x^p)$, for all $0\leq m-1$ and $u(x^p,y)$ is a unit in $k[[x^p,y]]$ and hence also in $k[[x,y]]$. In fact $m \geq 2$ since it assumed that $Q\in \tilde{C}$ is singular. Then $I_C=(y^m+h(x^p,y))$, where 
\[
h(x^p,y)=f_0(x^p)+f_1(x^p)y+\cdots + f_{m-1}(x^p)y^{m-1} \in (x,y)^{p+1}\subset k[[x,y]]
\]
and $m\geq 2$. Suppose that $m\geq p$. Then $m_P(C) \geq p$ and hence $p_a(C) \geq p$, which is impossible since by assumption 
$\mathrm{p}_a(C)\leq (p-1)/2$. Suppose that $m<p$.  Then write $p=sm+r$, $0<r<m$. After blowing up $P\in C$ and its infinitely near singular points $s$ times we see by using the adjunction formula that 
\begin{gather}\label{sec1-prop4-eq1}
2\mathrm{p}_a(C) \geq sm(m-1).
\end{gather}
Suppose that $m\geq (p+1)/2$. Then $m-1\geq (p-1)/2$ and hence  from the above inequality it follows that 
\[
\mathrm{p_a}(C)\geq \left(\frac{sm}{2}\right)\left(\frac{p-1}{2}\right)\geq \frac{p-1}{2},
\]
since $m\geq 2$. 

Suppose that $m < (p+1)/2$. Then also $r<m < (p+1)/2$. Then $p-r>(p-1)/2$ and hence
\begin{gather}\label{sec1-eq-111}
\mathrm{p}_a(C)\geq \frac{1}{2}sm(m-1)=(p-r)\frac{m-1}{2} \geq \left(\frac{p-1}{2}\right) \left(\frac{m-1}{2}\right).
\end{gather}
Suppose that $m\geq 3$. Then from the above inequality it follows that $\mathrm{p}_a(C)\geq (p-1)/2$. Suppose that $m=2$. Then  $s=(p-1)/2$ and $r=1$. Then from the equation~\ref{sec1-eq-111} it follows again that $\mathrm{p}_a(C)\geq (p-1)/2$.

\textbf{Case 2.} $Q\in \tilde{C}$ is smooth. Then  $C \rightarrow \tilde{C}$ is a $\mu_p$ or $\alpha_p$ torsor. Hence 
\[
\mathcal{O}_C=\frac{\mathcal{O}_{\tilde{C}}[t]}{(t^p-s)}
\]
where $s\in \mathcal{O}_{\tilde{C}}$. Let $x$ be local analytic coordinate of 
$\tilde{C}$ at $Q$. Then locally analytically at $Q\in \tilde{C}$, $\mathcal{O}_{\tilde{C}}=k[[x]]$ and $s=f(x)\in k[[x]]$. Moreover, since $P\in C$ is singular, $f(x)\in (x^2)$. Therefore
\[
\mathcal{O}_C=\frac{\mathcal{O}_{\tilde{C}}[t]}{(t^p-s)}=\frac{k[[x,t]]}{(t^p-f(x))}.
\]
Then one can write $f(x)=x^m u(x)$, where $u(x)$ is a unit in $k[[x]]$. If $m< p$ then $\sqrt[m]{u(x)}$ exists and therefore locally analytically at $P$, 
\[
\mathcal{O}_C\cong \frac{k[[x,y]]}{(t^p-x^m)}.
\]
If $p \leq m$ then since $k$ has characteristic $p$, the $\sqrt[m]{u(x)}$ does not always exist. But in this case $m_P(\mathcal{O}_{C,P})\geq p$ which is impossible since $\mathrm{p}_a(C)<p$. Hence $I_C=(t^p-x^m)$, $m\geq 2$. Then by using the same argument as in Case 1 it follows that $p_a(C)\geq (p-1)/2$, which is impossible.

So far it has been proved that every singular point of $C$ is a fixed point of $D$ as well. Hence $D$ lifts to the blow up of $X$ at any singular point of $C$. Then by repeating the previous arguments it follows that $D$ fixes every infinitely near singular point of $C$ as well.
\end{proof}

\begin{corollary}\label{sec1-cor2}
With assumptions as in Proposition~\ref{sec1-prop4}. Suppose in addition that $C$ is singular. Let $D_c$ be the vector field on $C$ induced by $D$. Suppose that $D_c\not=0$. Let $\bar{C}\rightarrow C$ be the normalization of $C$. Then $\bar{C}\cong \mathbb{P}^1_k$. Moreover
\begin{enumerate}
\item Suppose that $D^p=0$. Then $D$ has exactly one fixed point on $C$.
\item Suppose that $D^p=D$. Then $D$ has exactly two distinct points on $C$ (perhaps infinitely near).
\end{enumerate}
In particular, $C$ is rational.

\end{corollary}

\begin{proof}
By Proposition~\ref{sec1-prop4}, $D$ fixes the singular points of $C$ and all its infinitely near singular points as well. Hence $D_c$ lifts to a vector field 
$\bar{D}$ on the  normalization $\bar{C} \rightarrow C$. Considering that smooth curves of arithmetic genus greater or equal than 2 do not have nontrivial global vector fields, it follows that 
$\mathrm{p}_a(\bar{C})\leq 1$.  Suppose that $\bar{C}$ is a smooth elliptic curve. In this case $T_C=\mathcal{O}_C$ and hence the unique global vector field of $\bar{C}$ has no fixed points. This case is impossible since $\bar{D}$ fixes the preimages of the singular points of $C$. 
 Hence $\bar{C}=\mathbb{P}^1$. In this case $T_{\bar{C}}=\omega_{\mathbb{P}^1}^{-1}=\mathcal{O}_{\mathbb{P}^1}(2)$. Hence $\mathbb{P}^1$ has three linearly independent global vector fields $D_i$, $i=1,2,3$. These vector fields are induced from the homogeneous vector fields $D_1=x\frac{\partial}{\partial x}$, $D_2=x\frac{\partial}{\partial y}$ and $D_3=y\frac{\partial}{\partial x}$ of $k[x,y]$. Note that $D_1^p=D_1$ and $D_i^p=0$, $i=2,3$. Hence there are $a_i \in k$, $i=1,2,3$, such that $\bar{D}=a_1D_1+a_2D_2+a_3D_3$. 

\textbf{Claim:} $\bar{D}^p=\bar{D}$ if and only if $a_2=a_3=0$ and $a_1\in\mathbb{F}_p^{\ast}$, and $\bar{D}^p=0$ if and only if $a_1^2+4a_2a_3=0$. 

In order to show this restrict $\bar{D}$ to the standard affine cover of $\mathbb{P}^1$.

Let $U \subset \mathbb{P}^1$ be the open affine subset given by $y\not= 0$. Let $u=x/y$. Then an easy calculation shows that $D_1=u\frac{d}{du}$, $D_2=-u^2\frac{d}{du}$ and $D_3=\frac{d}{du}$. Therefore 
\[
\bar{D}=(-a_2u^2+a_1u+a_3)\frac{d}{du}
\]
in $U$. I will now show that this is additive if and only if $-a_2u^2+a_1u+a_3=0$ has either a double root or no roots and multiplicative if and only if $a_2=0$ and $a_1\in \mathbb{F}_p$. Suppose that the previous equation  has a double root, and hence $a_1^2+4a_2a_3=0$. Then after a linear automorphism of $k[u]$, $\bar{D}=au^2\frac{d}{du}$, $a\in k$. This can easily verified to be additive. Suppose on the other hand that $-a_2u^2+a_1u+a_3=0$ has either two distinct roots or only one simple root (hence $a_2=0$). Suppose that $a_2\not=0$ and hence it has two distinct roots. Then after a linear automorphism of $k[u]$, $\bar{D}=a(u^2+u)\frac{d}{du}$. Then an easy calculation shows that
\[
D^p(u^{p-1})=a^p(p-1)^p(u^p+u^{p-1})=-a^p(u^p+u^{p-1})\not= 0.
\]
Hence in this case $\bar{D}$ is neither additive or multiplicative. Hence $a_2=0$ and $\bar{D}=(a_1u+a_3)\frac{d}{du}$. Then $\bar{D}^p=a_1^{p-1}\bar{D}$. Hence $\bar{D}^p=\bar{D}$ if and only if $a_1^{p-1}=1$ and therefore if and only if $a_1\in\mathbb{F}_p$. 

Let $V$ be the affine open subset of $\mathbb{P}^1$ given by $x\not=0$. Let $v=y/x$. Then in $V$, $D_1=- v\frac{d}{dv}$, $D_2=\frac{d}{dv}$ and $D_3=-v^2\frac{d}{dv}$. Therefore
\[
\bar{D}=(-a_3v^2-a_1v+a_2)\frac{d}{dv}.
\]
Suppose that $\bar{D}$ is additive. Then similar arguments as before show that $a_1^2+4a_2a_3 =0$. Suppose that $\bar{D}$ is of multiplicative type. Then as before we get that $a_3=0$. This concludes the proof of the claim.

Suppose now that $\bar{D}$ is of multiplicative type. Then it has been shown that $\bar{D}=ax\frac{\partial}{\partial x}$, $a\in \mathbb{F}_p^{\ast}$. The fixed points of this are $[0,1]$ and $[1,0]$. In particular it has exactly two distinct fixed points. 

Suppose that $\bar{D}$ is of additive type. Then from the previous arguments it follows that $\bar{D}$ has a single fixed point (but with double multiplicity). 

Hence if $D^p=D$, then $D$ has at most two distinct points and if $D^p=0$ then it has just one.
\end{proof}

The next results will also be needed in the proof of the main theorem.

\begin{proposition}[Corollary 7.9~\cite{Ha77}]\label{sec1-prop5}
Let $X$ be an integral normal projective variety over an algebraically closed field $k$. Let $Y \subset X$ be a closed subscheme of $X$ which is the support of an effective ample divisor. Then $Y$ is connected.
\end{proposition}

\begin{proposition}\label{sec1-prop-6}
Let $f \colon Y \rightarrow X$ be a composition of $n$ blow ups starting from a smooth point $P \in X$ of a surface $X$. Let $C \subset X$ be an integral curve in $X$ passing through $P$ and let $m=m_Q(C)$ be the multiplicity of $C$ at $P\in C$. Then
\[
mK_Y-f^{\ast}C+C^{\prime}=mf^{\ast}K_X+\sum_{k=1}^n(km-a_1-a_2-\ldots -a_k)E_k,
\]
where $E_i$, $1\leq i \leq n$ are the $f$-exceptional curves, $C^{\prime}$ is the birational transform of $C$ in $Y$  and $0\leq a_i \leq m$, are nonnegative integers.
\end{proposition}
The proof of the proposition is by a simple induction on the number of blow ups $n$ and is omitted.

\begin{proposition}\label{sec1-prop-7}
Let $P\in S$ be a Duval singularity and let $C\subset S$ be a smooth curve such that $P\in S$. Let $f \colon S^{\prime}\rightarrow S$ be the minimal resolution of $P\in S$,  and $E_i$, $i=1,\dots, n$ be the $f$-exceptional curves. Let $C^{\prime}$ be the birational transform of $C$ in $S^{\prime}$ and $a_i>0$, $1\leq i \leq n$ be positive rational numbers such that
\[
f^{\ast}C=C^{\prime}+\sum_{i=1}^na_i E_i.
\]
Then
\begin{enumerate}
\item Suppose that $P\in S$ is of type $A_n$. Then $(n+1)C$ is Cartier in $S$ and $(n+1)a_i$ are positive integers $\leq n$, $i=1,\dots, n$.
\item Suppose that $P\in S$ is of type $D_n$. Then $4C$ is Cartier in $S$ and $4a_i$are integers $\leq n$, $i=1,\ldots, n$.
\item Suppose that $P\in S$ is of type $E_6$. Then $3C$ is Cartier in $S$ and $3a_i$are integers $\leq 6$, $i=1,\ldots, 6$.
\item Suppose that $P\in S$ is of type $E_7$. Then $2C$ is Cartier in $S$ and $2a_i$are integers $\leq 7$, $i=1,\ldots, 7$.
\end{enumerate}
\end{proposition} 
Notice that $P\in S$ cannot be of type $E_8$ because this singularity is factorial and hence there is no smooth curve passing through it.

The proof of this proposition is by a straightforward computation of the coefficients $a_i$ in $f^{\ast}C$ depending on the type of the singularity and the position of $C^{\prime}$ in the dual graph of the exceptional locus of the singularity and it is omitted. Similar computations can be found in~\cite[Proposition 4.5]{Tz03}.


\section{Set up and methodology of the proof of the main theorem.}\label{sec-2}

Let $X$ be a smooth canonically polarized surface defined over an algebraically closed field of characteristic $p>0$ with a nontrivial global vector field or equivalently with non smooth automorphism scheme. The main idea of the strategy for the proof of Theorems~\ref{intro-the-1},~\ref{intro-the-2} is to do one of the following:
\begin{enumerate}
\item Find an integral curve $C$ of $D$ on $X$ with the following properties: Its arithmetic genus $\mathrm{p}_a(C)$ is a function of $K_X^2$, $\mathrm{p}_a(\bar{C})\geq 1$, where $\bar{C}$ is the normalization of $C$, and such that $C$ contains some of the fixed points of $D$. Then by using the results of Section~\ref{sec-2}, if $\mathrm{p}_a(C)$ is small enough compared to the characteristic $p$, $D$ induces a vector field  on $C$ which lifts to $\bar{C}$. But this would be impossible since smooth curves of genus greater or equal than two have no nontrivial global vector fields and global vector fields on smooth elliptic curves do not have fixed points. This argument will allow us to conclude that if $K_X^2 < f(p)$, for some function $f(p)$ of $p$ then $X$ does not have any nontrivial global vector fields.
\item Find a positive dimensional family of integral curves $\{C_t\}$ of $D$ whose arithmetic genus is a function of $K_X^2$.  Then by using Tate's theorem~\cite{Sch09},~\cite{Ta52}  on the general member of the family, $(p-1)/2 <\mathrm{p}_a(C_t)$. This will make it  possible to get again a result as by the previous technique. 

\end{enumerate}

In order to achieve this, the following method will be used. It is based on a method initially used in~\cite{RS76} and then in~\cite{Tz17} but with different objectives.

Since $X$ has a nontrivial global vector field, then by~\cite[Proposition 4.1]{Tz17} $X$ has a nontrivial global vector field $D$ of either additive or multiplicative type which induces a  nontrivial $\alpha_p$ or 
$\mu_p$ action.  Let $\pi \colon X \rightarrow Y$ be the quotient. Then $Y$ is normal, $K_Y$ is $\mathbb{Q}$-Cartier and the local class groups of its singular points are p-torsion~\cite[Proposition 3.5]{Tz17b}. Consider now the following diagram

\begin{equation}\label{sec2-diagram-1}
\xymatrix{
     &       &    Y^{\prime}\ar[dl]_{g}\ar[dr]^{h} & \\
 X \ar[r]^{\pi} & Y &  & Z \\
}
\end{equation}
where $g \colon Y^{\prime} \rightarrow Y$ is the minimal resolution of $Y$ and  $\phi \colon Y^{\prime} \rightarrow  Z$ its minimal model.

Integral curves on $X$ will be found by choosing a suitable a reflexive sheaf $L$ on $Y$ such that either $\dim H^0(L)\geq 2$, in which case the pullbacks in $X$ of the divisors of $Y$ corresponding to the sections of $L$ will be integral curves of $D$, or $\dim H^0((\pi^{\ast}L) ^{[1]})\geq 2$ and then study the action of $D$ on $H^0((\pi^{\ast}L)^{[1]})$ exhibited in  Proposition~\ref{sec1-prop1}. The eigenvectors of this action will be  integral curves of $D$.

The proof of Theorems~\ref{intro-the-1},~\ref{intro-the-2} will distinguish cases with respect to whether $D$ has a divisorial part $\Delta$ or not and according to the Kodaira dimension $\kappa(Z)$ of $Z$. Then results from the classification of surfaces in positive characteristic will be heavily used~\cite{BM76},~\cite{BM77},~\cite{Ek88} and the geometry o $X$ and $Z$ will be compared by using diagram~\ref{sec2-diagram-1}. Moreover, since $\pi$ is a purely inseparable map, it induces an equivalence between the \'etale sites of $X$ and $Y$. Therefore $X$ and $Y$ have the same algebraic fundamental group, $l$-adic betti numbers and \'etale Euler characteristic. Then by using the fact that $g$ and $h$ are  birational it will be possible to calculate the algebraic fundamental group, $l$-adic Betti numbers and \'etale Euler characteristic of $X$ from those of $Z$.

Finally I collect some formulas and set up some terminology and notation that will be needed in the proof.

 Let $\Delta$ be the divisorial part of $D$. If $\Delta=0$ then we say that $\Delta$ has only isolated singularities. There is also the following  adjunction formula for purely inseparable maps~\cite[Corollary 1]{RS76} 
\begin{gather}\label{sec2-eq-2}
K_X=\pi^{\ast}K_Y+(p-1)\Delta.
\end{gather}

Let $F_i$, $i=1,\ldots, n$ be the $g$-exceptional curves and $E_j$, $j=1,\ldots,m$ be the $\phi$-exceptional curves. By~\cite[Lemma 5.1]{Tz17}, the $g$-exceptional curves $F_i$ are all rational (but perhaps singular). 

Taking into consideration  that  $g \colon Y^{\prime}\rightarrow Y$ is the minimal resolution of $Y$, we get the following adjunction formulas
\begin{gather}\label{sec2-eq-1}
K_{Y^{\prime}}+\sum_{i=1}^na_iF_i=g^{\ast}K_Y,\\
K_{Y^{\prime}}=h^{\ast}K_Z+\sum_{j=1}^mb_jE_j,\nonumber
\end{gather}
where $a_i \in\mathbb{Z}_{\geq 0}$, and $b_j>0$, $j=1,\ldots m$. Moreover since both $Y^{\prime}$ and $Z$ are smooth, $h$ is the composition of $m$ blow ups.

Note that the cases when $K_X^2\leq 2$ have been treated in ~\cite{Tz17}. Hence from now on it will be assumed that $K_X^2\geq 3$. Moreover, it will be assumed that $p\not=2,3$. For the purposes of this work this is not a serious restriction since general bounds of the type $K_X^2<f(p)$ are sought that guarantee the smoothness of $\mathrm{Aut}(X)$.

Finally, for the rest of the paper fix the notation of this section.


\section{Vector fields with only isolated singularities.}\label{sec-3}
Fix the notation as in Section~\ref{sec-2}. The main result of this section is the following.
\begin{theorem}\label{sec3-th-1}
Let $X$ be a smooth canonically polarized surface defined over an algebraically closed field of characteristic $p>0$. Suppose that $X$ admits a nontrivial global vector field $D$ such that $D^p=0$ or $D^p=D$. Assume moreover that $D$ has only isolated singularities. Then
\begin{enumerate}
\item Suppose that $D$ is of multiplicative type. Then if $K_X^2<(p-3)/144$, then $X$ is unirational, $b_1(X)=0$ and $|\pi_1(X)|\leq 2$.
\item Suppose that $D$ is of additive type. Then if $8(K_X^2)^3+4(K_X^2)^2<p-3$ and $K_X^2<\frac{p-3}{2\cdot 44 \cdot 45}$ then $X$ is unirational, $b_1(X)=0$ and $|\pi_1(X)|\leq 2$. In particular, this happens if 
\[
K_X^2 < \mathrm{min} \left\{\frac{1}{2}\sqrt[3]{\frac{p-3}{2}}, \frac{p-3}{2\cdot 44\cdot 45} \right\}.
\]
\end{enumerate} 
\end{theorem}

\begin{proof}

Suppose then that $D$ is a global vector field on a smooth canonically polarized surface $X$ with only isolated singularities, i.e., $\Delta=0$. Then from the equation~\ref{sec2-eq-2} it follows that 
$K_X=\pi^{\ast}K_Y$. 
 Therefore, since $K_X$ is ample, $K_Y$ is ample as well. In particular, if $K_X^2<p$, then $Y$ is singular. If that was not true then $K_X^2=pK_Y^2\geq p$.

Next consider cases with respect to the Kodaira dimension $\kappa(Z)$ of $Z$.

\subsection{Suppose that $\kappa(Z)=2$.} In this case I will show that  $K_X^2>(p-3)/20$.

According to~\cite[Theorem 1.20]{Ek88}, the linear system $|4K_Z|$ is very ample. Let $W\in|4K_Z|$ a smooth and connected member which does not go through the points blown up by $h$ in the diagram~\ref{sec2-diagram-1}. Then by the adjunction formula, $p_a(W)=10K_Z^2+1$. Then combining the equations~\ref{sec2-eq-1} it follows that
\begin{gather}\label{sec3-eq-1}
g^{\ast}(4K_Y)=4K_{Y^{\prime}}+4\sum_{i=1}^na_iF_i=h^{\ast}(4K_Z)+4\sum_{j=1}^mb_jE_j+4\sum_{i=1}^na_iF_i\sim \\
W^{\prime}+4\sum_{j=1}^mb_jE_j+4\sum_{i=1}^na_iF_i,
\end{gather}
where $W^{\prime}=h^{\ast}W=h_{\ast}^{-1}W$ is the birational transform of $W$ in $Y^{\prime}$. By pushing down to $Y$ we get that 
\begin{gather}\label{sec3-eq-2}
4K_Y\sim \tilde{W}+4\sum_{i=1}^m b_i\tilde{E}_i,
\end{gather}
where $\tilde{F}_i=g_{\ast}F_i$, $1\leq i \leq m$. Note that since $Y^{\prime}$ is the minimal resolution of $Y$, $g$ does not contract the $-1$ $h$-exceptional curves. Hence if $h$ is not an isomorphism then $g_{\ast}\sum_{i=1}^mE_i \not= 0$. 

Let $\hat{W}=\pi^{\ast}\tilde{W}$ and $\hat{E}_i=\pi^{\ast}\tilde{E}_i$, $1\leq i \leq m$. Then since $K_X=\pi^{\ast}K_Y$ it follows from the equation~\ref{sec3-eq-2} that
\begin{gather}\label{sec3-eq-3}
4K_X\sim \hat{W} +4\sum_{i=1}^m b_i\hat{E}_i.
\end{gather}
Suppose that $\hat{W}= \pi^{\ast}\tilde{W}$ is not irreducible. Then by~\cite{RS76}, it must be that $\pi^{\ast}\tilde{W}=p\bar{W}$, and $\bar{W}\rightarrow \tilde{W}$ birational. But then in this case, since $K_X$ is ample, it follows from  the equation~\ref{sec3-eq-3} that $K_X^2>p/4>(p-3)/20$, as claimed.

Suppose that $\hat{W}$ is irreducible. This means that $\hat{W}$ is an integral curve of $D$ and that $\hat{W}\rightarrow \tilde{W}$ is purely inseparable of degree $p$. In fact, $D$ induces a vector field on $\hat{W}$ which is nontrivial since $D$ has only isolated singularities. Then $\tilde{W}$ is the quotient of 
$\hat{W}$ by the induced $\alpha_p$ or $\mu_p$ action on $\hat{W}$. Let $\mu\colon \bar{W} \rightarrow \hat{W}$ be the normalization of $\hat{W}$.

\textbf{Claim:} $\bar{W} \cong W$ and therefore $\mathrm{p}_a(\bar{W})=10K_Z^2+1\geq 11$. 

Let $F\colon \hat{W} \rightarrow \hat{W}^{(p)}$ be the geometric Frobenius. Then there exists a factorization
\[
\xymatrix{
\hat{W}\ar[rr]^{F} \ar[dr]^{\pi} & & \hat{W}^{(p)}\\
&      \tilde{W} \ar[ur]^{\delta} &
}
\]
Since $F$ and $\pi$ are purely inseparable of degree $p$, it follows that $\delta $ is birational. Now $g \colon W^{\prime} \rightarrow \tilde{W}$ is also birational and $W^{\prime}\cong {W}$ is smooth. Hence $W$ is isomorphic to the normalization of $\hat{W}^{(p)}$. Then by the properties of geometric Frobenius, there exists a commutative diagram
\[
\xymatrix{
\bar{W}\ar[d]^{F^{(p)}}\ar[r]^{\mu} & \hat{W} \ar[d]^{F^{(p)}}\\
\bar{W}^{(p)} \ar[r]^{\sigma} & \hat{W}^{(p)}
}
\]
From the above diagram it follows that $\sigma$ is birational. Therefore $\bar{W}^{(p)}$ is the normalization of $ \hat{W}^{(p)}$. Hence 
$\bar{W}^{(p)}\cong W$ and therefore
\[
\mathrm{p}_a(\hat{W}) \geq \mathrm{p}_a(\bar{W})=\mathrm{p}_a(\bar{W}^{(p)})=\mathrm{p}_a(W) =10K_Z^2+1\geq 11,
\]
as claimed.
 
Now from the equation~\ref{sec3-eq-3} it follows that $K_X\cdot \hat{W} < 4K_X^2$. Then, since $K_X$ is ample, it follows from the Hodge index theorem that 
\[
\hat{W}^2K_X^2 < (K_X\cdot W)^2< 16 (K_X^2)^2
\]
and therefore $\hat{W}^2<16K_X^2$. Hence from the adjunction formula we get that
\[
11 \leq \mathrm{p}_a(\hat{W}) \leq 10K_X^2+1.
 \]
Then by Proposition~\ref{sec1-prop4} it follows that if $10K_X^2+1 <(p-1)/2$, or equivalently if $K_X^2<(p-3)/20$, $D$ fixes the singular points of $\hat{W}$ and therefore its restriction on $\hat{W}$ lifts to the normalization $\bar{W}$ of $\hat{W}$. But since $p_a(\bar{W})\geq 2$, $\bar{W}$ does not have any nontrivial global vector fields. But this implies that the restriction of $D$ on $\hat{W}$ is zero, which is impossible since $D$ has only isolated singularities. Hence we must have $D=0$ unless $K_X^2>(p-3)/20$.

\subsection{Suppose that $\kappa(Z)=1$.} In this case I will show that $K_X^2>(p-3)/42$.

Since $\kappa(Z)=1$, it is well known that $Z$ admits an elliptic fibration $\phi \colon Z \rightarrow B$, where $B$ is a smooth curve. Then one can write
\begin{gather}\label{sec3-eq-4}
R^1\phi_{\ast}\mathcal{O}_Z=L\oplus T,
\end{gather}
where $L$ is an invertible sheaf on $B$ and $T$ is a torsion sheaf. 

\textbf{Claim:} $B\cong \mathbb{P}^1$. Moreover, if $K_X^2< p/14$, then $T=0$. 

The $g$-exceptional curves are rational. Hence if at least one of them is not contracted to a point by $\phi \circ h$, then $B$ is dominated by a rational curve and  hence it is isomorphic to $\mathbb{P}^1$. Suppose that every $g$-exceptional curve is contracted to a point by $\phi \circ h$. Then by looking at diagram~\ref{sec2-diagram-1} we see that there exists factorizations
\[
\xymatrix{
    & Y \ar[dr]^{\psi} & \\
    X \ar[ur]^{\pi} \ar[rr]^{\sigma} & & B
    }
\]
such that the general fiber of $\psi$ is an elliptic curve. Then let $Y_b=\psi^{-1}(b)$ be the general fiber. Then $K_Y\cdot Y_b=0$ and therefore, 
\[
K_X\cdot \pi^{\ast} Y_b=\pi^{\ast}K_Y \cdot \pi^{\ast}Y_b=p K_Y \cdot Y_b=0.
\]
But this is impossible since $K_X$ is ample. Therefore there must be a $g$-exceptional curve not contracted to a point by $\phi\circ h$ and hence $B\cong \mathbb{P}^1$.

Suppose now that $T\not= 0$. Let  $b \in T$. Then $\phi^{-1}(b)=pmW$, $m>0$ and $W$ is an idecomposable fiber~\cite{KU85}. Moreover $|14KZ|$ defines the fibration $\phi$~\cite{KU85}. Hence $14K_Z \sim \nu F$, where $F$ is a general fiber of $\phi$ and hence a smooth elliptic curve (if $p\not= 2,3$.).  Then 
$F\sim \phi^{-1}(b)=pmW$. and hence $14K_Z\sim pm\nu W$. Then by pulling up to $Y^{\prime}$ it follows that
\[
14h^{\ast}K_Z=pm\nu W^{\prime} +p(\sum_{i=1}^m c_i E_i).
\]
If $h$ blows up a point of $W$ then $c_i>0$ and $14h^{\ast}K_Z$ has a component corresponding to a $-1$ $h$-exceptional curve with coefficient divisible by $p$. Considering that the $-1$ $h$-exceptional curves do not contract by $g$, we see that in any case (if $h$ blows up a point on $W$ or not) that, after pushing down to $Y$, $14K_Y \sim p \tilde{W} + B $, for some divisor $\tilde{W}$ (either the birational transform of $W$ or the image of a $-1$ $h$-exceptional curve.  Therefore by pulling up to $X$ and since $K_X=\pi^{\ast} K_Y$,
\[
14K_X \sim p\pi^{\ast}\tilde{W} +\pi^{\ast}B.
\]
But from this it follows that $K_X^2>p/14$, a contradiction. This concludes the proof of the claim.

Next consider cases with respect to $p_g(Z)$.

\textbf{Case 1.} Suppose that $p_g(Z)\geq 2$. In this case I will show that $K_X^2\geq (p-3)/2$.

Since $\dim H^0(\mathcal{O}_Z(K_Z))\geq 2$, it follows that $K_Z \sim mF$, where $m>0$ is a positive integer and $F$ is a general fiber of $\phi$ (note that since $K_Z\cdot F=0$, for any fiber of $\phi$, any member of $|K_Z|$ is supported on a fiber of $\phi$). Since it is assumed that $p \not= 2,3$, $F$ is a smooth elliptic curve. Then by pulling up to $Y^{\prime}$ and pushing down to $Y$ it follows that 
\begin{gather}\label{sec3-eq-5}
K_Y=m\tilde{F} +B,
\end{gather}
where $\tilde{F}$ is the birational transform of $F$ on $Y$ and $B$ some effective divisor. Let $\hat{F}=\pi^{\ast}\tilde{F}$. For the same reasons as in the proof of the case when $\kappa(Z)=2$, we see that $\hat{F}$ is irreducible and its normalization is isomorphic to $F$ and hence it is a smooth elliptic curve. Hence $\hat{F}$ is an integral curve of $D$ and therefore $D$ induces a  nontrivial (since $D$ has only isolated singularities) vector field $\hat{D}$ on $\hat{W}$. Now pulling back the equation~\ref{sec3-eq-5} to $X$ we get that 
\[
K_X=m\hat{F}+\pi^{\ast}B.
\]
From this and the Hodge index theorem it immediately follows that $K_X\cdot \hat{F} \leq K_X^2$ and $\hat{F}^2 < K_X^2$. 
Hence $\mathrm{p}_a(\hat{F})<K_X^2+1$. Then if $K_X^2+1 <(p-1)/2$, equivalently if $K_X^2<(p-3)/2$, then $\mathrm{p}_a(\hat{F})< (p-1)/2$ and therefore from Proposition~\ref{sec1-prop4} that $\hat{D}$ fixes the singular points (if any)  of $\hat{F}$ and therefore it lifts to a vector field $\bar{D}$ of its normalization $\bar{F}$. However, since there exists at least one $g$-exceptional curve dominating $\mathbb{P}^1$. it easily follows that $\tilde{F}$ goes through at least one singular point of $Y$ and therefore $D$ has at least one fixed point on $\hat{F}$. But then $\bar{D}$ has a fixed point, which is impossible since $\bar{E}$ is a smooth elliptic curve.

\textbf{Case 2.} Suppose that $\mathrm{p}_g(Z) \leq 1$. In this case I will show that $K_Z^2\geq (p-3)/42$.

From the Noether's formula on $Z$~\cite[Theorem 5.1]{Ba01}
\begin{gather}\label{sec3-eq-6}
10-8h^1(\mathcal{O}_Z)+12p_g(Z)=K_Z^2+b_2(Z)+2(2h^1(\mathcal{O}_Z)-b_1(Z))=\\
b_2(Z)+2(2h^1(\mathcal{O}_Z)-b_1(Z))\nonumber
\end{gather}
it easily follows~\cite[Page 113]{Ba01} that if $p_g(Z) \leq 1$, then the only numerical solutions to the equation~\ref{sec3-eq-6} are the following:
\begin{enumerate}
\item $p_g(Z)=0$, $\chi(\mathcal{O}_Z)=0$, $b_1(Z)=2$.
\item $p_g(Z)=0$, $\chi(\mathcal{O}_Z)=1$, $b_1(Z)=0$.
\item $p_g(Z)=1$, $\chi(\mathcal{O}_Z)=2$, $b_1(Z)=0$.
\item $p_g(Z)=1$, $\chi(\mathcal{O}_Z)=1$, $b_1(Z)=2$.
\item $p_g(Z)=1$, $\chi(\mathcal{O}_Z)=1$, $b_1(Z)=0$.
\item $p_g(Z)=1$, $\chi(\mathcal{O}_Z)=0$, $b_1(Z)=2$.
\item $p_g(Z)=1$, $\chi(\mathcal{O}_Z)=0$, $b_1(Z)=4$.
\end{enumerate}
Note that by~\cite[Lemma 3.5]{KU85} the last case is not possible. Consider next each one of the cases separately. I will only consider the first two cases. The rest are similar and are omitted.

\textbf{Case 2.1.} Suppose that $p_g(Z)=\chi(\mathcal{O}_Z)=0$ and $b_1(Z)=2$.  In this case I will show that $K_X^2\geq (p-3)/42$.

By Igusa's formula~\cite{IG60} it follows that the fibers of $\phi \colon Z \rightarrow \mathbb{P}^1$ are either smooth elliptic curves or of  type $mE$, where $m$ is a positive integer and $E$ an elliptic curve (singular or smooth). Also note that $\phi$ must have multiple fibers or else $Z$ cannot have Kodaira dimension 1. 

I will next show that in fact $E$ is a smooth elliptic curve. Indeed. Since $b_1(Z)=2$ it follows that $\dim \mathrm{Alb}(Z)=1$. Hence $\mathrm{Alb}(Z)$ is a smooth elliptic curve. Let then $\psi \colon Z \rightarrow \mathrm{Alb}(Z)$ be the Albanese map. Then there exist the following two maps
\[
\xymatrix{
Z\ar[r]^{\psi}\ar[d]^{\phi} & \mathrm{Alb}(Z)\\
\mathbb{P}^1 & \\
}
\]
Suppose that $mE$ is a multiple fiber of $\phi$. Suppose also that $E$ is a rational elliptic curve. Then $E$ cannot dominate $\mathrm{Alb}(Z)$ and hence it must contract by $\psi$. Hence all fibers of $\phi$ contract by $\psi$. But then there would be a nontrivial map $\mathbb{P}^1 \rightarrow \mathrm{Alb}(Z)$, which is impossible. Hence $E$ is a smooth elliptic curve. 

It is well known~\cite[Theorem 8.11]{Ba01} that the linear system $\nu| K_Z|$, $\nu\in \{4,6\}$ contains a strictly positive divisor. Then $\nu K_Z\sim sE$, where $s>0$ is a positive integer and $E$ is a smooth elliptic curve. Let $E^{\prime}=h_{\ast}^{-1}E$ be the birational transform of $E$ in $Y^{\prime}$. Then $E^{\prime}$ is a smooth elliptic curve and since the $g$-exceptional curves are all rational, it follows that $E^{\prime}$ does not contract by $g$. Therefore by pulling up to $Y^{\prime}$ and then pushing down to $Y$ we get that 
\[
\nu K_Y \sim m\tilde{E} +B,
\]
where $B$ is an effective divisor on $Y$. Hence by pulling up to $X$ we get that
\begin{gather}\label{sec3-eq-7}
\nu K_X \sim m\hat{E} +\pi^{\ast} B.
\end{gather}
As in the previous cases we see that if $K_X^2< p/\nu$, $\hat{E}$ is irreducible and therefore is an integral curve of $D$ whose normalization $\bar{D}$ is a smooth elliptic curve. Now from the equation~\ref{sec3-eq-7} and by using the Hodge index theorem we get that $K_X \cdot \hat{E} < \nu K_X^2$ and that 
$\hat{E}^2<\nu^2 K_X^2$. Therefore from the adjunction formula it follows that
\[
\mathrm{p}_a(\hat{E}) < \frac{{\nu}(\nu+1)}{2}K_X^2 +1.
\]
 Hence if 
 \[
 K_X^2< \frac{p-3}{2}\cdot \frac{2}{\nu(\nu+1)},
 \]
 then $\mathrm{p}_a(\hat{E}) <(p-1)/2$. Considering that $\nu \in\{4,6\}$, the above inequality holds if $K_X^2<(p-3)/42$. Therefore in this case, from Proposition~\ref{sec1-prop4} it follows that 
 the restriction $\hat{D}$ of $D$ on $\hat{E}$ fixes the singular points of $\hat{E}$ and hence lifts to its normalization $\bar{E}$. However, for the same reasons as before, $\hat{E}$ contains some fixed points of
  $D$. Therefore the lifting of $D$ on $\bar{E}$ has fixed points, which is a contradiction.

\textbf{Case 2.2.} Suppose that $p_g(Z)=0$, $\chi(\mathcal{O}_Z)=1$, $b_1(Z)=0$. In this case I will show that $K_X^2\geq (p-3)/42$

\textbf{Claim:} $\dim |6K_Z| \geq 1$.

Let $F_{t_i}=m_iP_i$, $t_i\in \mathbb{P}^1$, $i=1,\ldots , r$ be the multiple fibers of $\phi$. Since $T=0$, they are all tame. Then by the canonical bundle 
formula~\cite[Theorem 7.15 and Page 118]{Ba01} we get that
\begin{gather}\label{sec3-eq-8}
\dim |nK_Z|=n(-2+\chi(\mathcal{O}_Z))+\sum_{i=1}^r \left[ \frac{n(m_i-1)}{m_i}\right]=-n+\sum_{i=1}^r \left[ \frac{n(m_i-1)}{m_i}\right],
\end{gather}
where for any $m\in \mathbb{N}$, $[m]$ denotes its integer part. Also, in the notation~\cite[Remark 8.3]{Ba01} if, 
\[
\lambda(\phi)=-1+\sum_{i=1}^r\frac{m_i-1}{m_i},
\]
Then $\kappa(Z)=1$ if and only if $\lambda(\phi)>0$. Hence $\phi$ has at least two multiple fibers.

Suppose that $\phi$ has at least three multiple fibers, i.e., $r\geq 3$ and $m_i\geq 2$. Then for every $1 \leq i \leq r$, 
\[
\left[ 6(1-\frac{1}{m_i})\right] \geq \left[\frac{6}{2}\right] =3.
\]
Then from the equation~\ref{sec3-eq-8} it follows that $\dim|6K_Z| \geq -6+3\cdot 3=3$.

Suppose that $\phi$ has exactly two multiple fibers with multiplicities $m_1$ and $m_2$. Then in order to have $\lambda(\phi)>0$, at least one of them must be greater or equal than $3$. Say $m_1\geq 3$ and $m_2\geq 2$. Then from the equation~\ref{sec3-eq-8} it follows that
\[
\dim|6K_Z| =-6+\left[6(1-\frac{1}{m_1}\right] +\left[6(1-\frac{1}{m_2}\right] \geq -6+\left[ 6\cdot \frac{2}{3}\right]+\left[ 6\cdot \frac{1}{2}\right]=1.
\]
Hence $6K_Z\sim mE$, where $m>0$ is a positive integer and $E$ is a smooth elliptic curve. By repeating now the argument used in Case 2.1 we see that this is impossible if $K_X^2<(p-3)/42$. This concludes the study of the case when $\kappa(Z)=1$.

\subsection{Suppose that $\kappa(Z)=0$.} In this case I will show that:
\begin{enumerate}
\item Suppose that $D$ is of multiplicative type. Then if $K_X^2<(p-3)/144$, then $X$ is unirational, $b_1(X)=0$ and $|\pi_1(X)|\leq 2$.
\item Suppose that $D$ is of additive type. Then if $8(K_X^2)^3+4(K_X^2)^2<p-3$ and $K_X^2<\frac{p-3}{2\cdot 44 \cdot 45}$ then $X$ is unirational, $b_1(X)=0$ and $|\pi_1(X)|\leq 2$. In particular, this happens if 
\[
K_X^2 < \mathrm{min} \left\{\frac{1}{2}\sqrt[3]{\frac{p-3}{2}}, \frac{p-3}{2\cdot 44\cdot 45} \right\}.
\]
\end{enumerate}
According to the classification of surfaces~\cite{BM76},~\cite{BM77}, $Z$ is one of the following: An abelian surface, a K3 surface, an Enriques surface, an elliptic or quasi-elliptic surface.

\textbf{Case 1.} Suppose that $Z$ is an abelian surface. Then every $g$-exceptional curve is also $h$-exceptional since every $g$-exceptional curve is rational and there do not exist nontrivial maps from a rational curve to an abelian surface. Hence there exists a factorization
\begin{gather}\label{sec3-eq-9}
\xymatrix{
Y^{\prime}\ar[r]^g\ar[d]_{\phi} & Y \ar[dl]^{\theta}\\
Z & \\
}
\end{gather}
Let $B_j$, $j=1,\ldots, r$ be the $\theta$-exceptional curves. Then one can write 
\[
K_Y=\theta^{\ast}K_Z +\sum_{j=1}^r\gamma_j B_j.
\]
Now from the diagram~\ref{sec3-eq-9} and the equations~\ref{sec2-eq-1}, one can easily see that $\gamma_j \geq 0$, for all $1\leq j \leq r$. But then, since $\{B_j, \; 1\leq j \leq r\}$ is a contractible set of curves, it easily follows that 
\[
K_Y^2=\left(\sum_{j=1}^r\gamma_j B_j\right)^2\leq 0,
\]
which is impossible since $K_Y$ is ample. Therefore $Z$ cannot be an abelian surface.

\textbf{Case 2.} Suppose that $Z$ is an elliptic or quasi-elliptic surface. I will show that this case is also impossible. It is well known that if $Z$ is quasi-elliptic,  then $b_1(Z)=2$~\cite{BM77} and hence $\dim \mathrm{Alb}(Z)=1$. Then the morphism $\phi \colon Z\rightarrow \mathrm{Alb}(Z)$ is an elliptic fibration~\cite{BM77}. Since every $g$-exceptional curve is rational, they must be contracted to points in $\mathrm{Alb}(Z)$. Hence there exists a factorization
\[
\xymatrix{
 & Y \ar[dr]^{\tilde{\phi}} & \\
 X \ar[ur]^{\pi}\ar[rr]^{\psi} & & \mathrm{Alb}(Z)
 }
 \]
 The general fiber $Y_b$ of $\tilde{\psi}$ is a smooth elliptic curve. Hence $K_Y \cdot Y_b=0$. hence
 \[
 K_X\cdot \pi^{\ast}Y_b=\pi^{\ast}K_Y\cdot \pi^{\ast}Y_b=pK_Y\cdot Y_b=0,
 \]
 which is impossible since $K_X$ is ample. Hence $Z$ can be either a K3 surface or an Enriques surface.
 
 \textbf{Case 3.} Suppose that $Z$ is a $K3$ surface. Consider now two cases with respect to whether $D$ is of multiplicative or additive type.
 
 \textbf{Case 3.1.} Suppose that $D$ is of multiplicative type, i.e., $D^p=D$. In this case I will show that if $K_X^2<(p-3)/72$, then $X$ is unirational and simply connected.
 
 By~\cite[Corollary 1.8]{Ek88}, it follows that $\dim|2K_X|\geq 2$. Moreover from~\cite[Theorem 1.20]{Ek88}, $|3K_X|$ is base point free. Also, since $K_X=\pi^{\ast}K_Y$,  by Proposition~\ref{sec1-prop1}, there exists $k$-linear maps
 \begin{gather}\label{sec3-eq-10}
 D_2^{\ast} \colon H^0(\mathcal{O}_X(2K_X) )\rightarrow H^0(\mathcal{O}_X(2K_X)),\\
  D_3^{\ast} \colon H^0(\mathcal{O}_X(3K_X)) \rightarrow H^0(\mathcal{O}_X(3K_X)).\nonumber
\end{gather}
Moreover, since $D^p=D$, both maps are diagonalizable (with eigenvalues in the set $\{0,1,\ldots, p-1\}$) and their eigenvectors correspond to integral curves of $D$. Let 
\begin{gather}\label{section3-eq-11}
 H^0(\mathcal{O}_X(3K_X))=\oplus_{i=1}^{k} V(\lambda_i),
 \end{gather}
 the decomposition of $H^0(\mathcal{O}_X(3K_X))$ in eigenspaces of $D_3^{\ast}$, where $\lambda_i \in\{0,1,\ldots, p-1\}$, $1\leq i \leq k$.
 
 Suppose that $\dim |3K_X|=m$. Let  $Z_i$, $i=1,\ldots, m$ be a basis of $|3K_Z|$ consisting of eigenvectors of $D_3^{\ast}$. Then since $Z_i$ are eigenvectors of $D_3^{\ast}$, $D$ induces nontrivial vector fields on each $Z_i$. Moreover, from 
 Corollary~\ref{sec1-cor1}, if $K_X^2<p/3$, then $D$ restricts to every reduced and irreducible component of $Z_i$, for all $1\leq i \leq m$.
 
 Since $Z$ is a $K3$ surface, $\omega_Z\cong \mathcal{O}_Z$. Hence from the equations~\ref{sec2-eq-1} it follows that 
 \begin{gather}\label{sec3-eq-12}
 K_Y=\sum_{j=1}^sb_j \tilde{E}_j,
 \end{gather}
 where $\tilde{E}_j$ is the birational transform in $Y$ of the $h$-exceptional curves not contracted by $g$. In particular $p_g(Y)\not=0$ and hence 
 $p_g(X)\not=0$. 
 Let $C=\pi^{\ast}\tilde{E}$, where $\tilde{E}$ is any irreducible component of $K_Y$ in the analysis~\ref{sec3-eq-12}. Then if $K_X^2<p$, $C$ is irreducible and hence is an integral curve of $D$. Moreover 
 \[
 K_X\cdot C=\pi^{\ast}K_Y \cdot \pi^{\ast} \tilde{E} =pK_Y\cdot \tilde{E}\leq pK_Y^2=K_X^2.
 \]
 Moreover, from the Hodge index theorem, $C^2\leq K_X^2$ as well and therefore $p_a(C)\leq K_X^2+1$. Now from Corollary~\ref{sec1-cor2} it follows that if 
 $K_X^2+1< (p-1)/2$, equivalently $K_X^2<(p-3)/2$, $D$ fixes the singular point of $C$ and lifts to its normalization. Therefore $C$ is a rational curve and $D$ has exactly two fixed points on $D$ (with the possibility to be infinitely near points). Let $P_1, P_2$  be the fixed points of $D$ on $C$.
 
Let $1\leq i \leq m$ be such that $C$ is not an irreducible component of $Z_i$.  Since $K_X$ is ample, it follows that $C\cdot Z_i>0$. For the same reason, 
$Z_i \cdot Z_j >0$ and therefore $Z_i \cap Z_j \not=\emptyset, $ for all $1\leq i,j \leq m$. Let now again $i$ be such that $C$ is not an irreducible component of $Z_i$. Let $A$ be an irreducible and reduced  component of $Z_i$. Then from the definition of $C$ and $Z_i$, it follows that \[
C \cdot A \leq C\cdot Z_i =3K_X \cdot C \leq 3K_X^2.
\]
Hence by Corollary~\ref{sec1-cor0}, if $K_X^2<p/3$, then every point of intersection of $A$ and $C$ is a fixed point of $D$. In particular, every point of intersection of $C$ and $Z_i$ is a fixed point of $D$ (in the case $C$ is not a component of $Z_i$).
 
 Suppose that $P_1=P_2$. Let $1\leq i \leq m$. Then either $C$ is a component of $Z_i$ or $(Z_i \cap C)_{\mathrm{red}} =\{P_1\}$, for all $1\leq i\leq m$. But this implies that $P_1$ is a base point of $|3K_X|$, which is impossible. Hence $P_1\not= P_2$. For the same reason, it is not possible that either 
 $(Z_i \cap C)_{\mathrm{red}} =\{P_1\}$, for all $i$ or  $(Z_i \cap C)_{\mathrm{red}} =\{P_2\}$, for all $i$. Therefore there exist indices $1\leq i\not= j \leq m$, such that $(Z_i\cap C)_{\mathrm{red}} =\{P_1\}$ and 
 $(Z_j \cap C)_{\mathrm{red}}=\{P_2\}$. But then, since $Z_i \cap Z_j \not= \emptyset$,  the curve $W=Z_i +Z_j + C$ contains loops. Let $\tilde{Z}_i=\pi(Z_i)$, 
 $\tilde{Z}_j=\pi(Z_j)$. Then $\tilde{W}=\tilde{Z}_i\ + \tilde{Z}_i\ + \tilde{E}$ is a curve whose reduced curve $\tilde{W}_{\mathrm{red}}$ contains loops. Hence 
 $\dim H^1(\mathcal{O}_{\tilde{W}_{\mathrm{red}}})\geq 1$ and hence  $\dim H^1(\mathcal{O}_{\tilde{W}_{\mathrm{red}}})\geq 1$ as well.
  
 Now since $Z$ is a $K3$ surface, it follows that $H^1(\mathcal{O}_Z)=0$ Hence $H^1(\mathcal{O}_{Y^{\prime}})=0$ and therefore from the Leray spectral sequence it follows that $H^1(\mathcal{O}_Y)=0$. Then from the exact sequence
 \[
 0 \rightarrow \mathcal{O}_Y(-\tilde{W}) \rightarrow \mathcal{O}_Y \rightarrow \mathcal{O}_{\tilde{W}} \rightarrow 0
 \]
 we get the exact sequence in cohomology
 \[
 \cdots 0=H^1(\mathcal{O}_Y) \rightarrow H^1(\mathcal{O}_{\tilde{W}}) \rightarrow H^2(\mathcal{O}_Y(-\tilde{W}) ) \rightarrow H^2(\mathcal{O}_Y) \rightarrow 
 H^2(\mathcal{O}_{\tilde{W}} )=0.
 \]
 Considering now that $h^1(\mathcal{O}_{\tilde{W}})\geq 1$,  $H^2(\mathcal{O}_Y(-\tilde{W}) )=H^0(\mathcal{O}_Y(\tilde{W}+K_Y))$, $H^2(\mathcal{O}_Y)=H^0(\mathcal{O}_Y(K_Y))$  and that $p_g(Y)\not= 0$, it follows that 
 \begin{gather}\label{sec3-eq-13}
 \dim H^0(\mathcal{O}_Y(\tilde{W}+K_Y))\geq 2.
 \end{gather}
 Now $\pi^{\ast}(\tilde{W}+K_Y)=W+K_Y\sim 8K_X$. Let 
 \[
 D_8^{\ast} \colon H^0(\mathcal{O}_X(8K_X)) \rightarrow  H^0(\mathcal{O}_X(8K_X))
 \]
 the $k$-linear map from Proposition~\ref{sec1-prop1}. The equation~\ref{sec3-eq-13}  says that $\dim V(0)\geq 2$, where $V(0)$ is the eigenspace of $0$ $D_8^{\ast}$. Therefore there exists a family of integral curves in $|8K_X|$ of dimension at least 1. Let $B$ be the general member of the family. This is defined over a non algebraically closed field $L$. Then from the adjunction formula it follows that 
 $\mathrm{p}_a(B)=36K_X^2+1$. Then by Propositions~\ref{sec1-prop2},~\ref{sec1-prop3},~\ref{sec1-prop4}, if $K_X^2<(p-3)/72$, then $\mathrm{p}_a(B) <(p-1)/2$ and hence $D$ restricts to the irreducible components of $B$ and also fixes their singular points. Notice also that  if $B=\pi^{\ast}\tilde{B}$, where $\tilde{B}\in|\tilde{W}+K_Y|$, then if $K_X^2<p/8$, $\tilde{B}$ goes through the singular points of $Y$ because if not then 
 from the equation $K_X=\pi^{\ast}K_Y$ follows that $8K_X^2=K_X \cdot B=\pi^{\ast}K_Y \cdot \pi^{\ast} \tilde{B}=p K_Y \cdot \tilde{B} \geq p$, if $\tilde{B}$ was contained in the smooth locus of $Y$. Hence if $K_X^2<p/8$, then $B$ contains fixed points of $D$.

 Consider cases with respect to the possibilities of $B$.
 
 \textit{Case 1.} Suppose that $B=sB^{\prime}$, where $B^{\prime}$ is an integral curve. 
 
 Suppose that $B^{\prime}$ is smooth curve over $L$. Then $\mathrm{p}_a(B^{\prime})\geq 2$ and hence the restriction of $D$ on $B^{\prime}$ must be zero. But this is not possible since $B$ is general and in fact $D$ does not have a divisorial part. 
 
 Suppose that $B^{\prime}$ is not smooth. Let $\bar{B}$ be its normalization over $L$. Then $D$ lifts to a vector field $\bar{D}$ on $\bar{B}$. Suppose that 
 $\bar{B}$ is regular but not smooth. Let then $L\subset M$ be an extension of $L$ such that $\bar{B}\otimes_L M$ is not regular and let $\hat{B}$ be the normalization of $\bar{B}\otimes_L M$ in $M$. Then by Tate's theorem, $(p-1)/2$ divides $\mathrm{p}_a(\bar{B})-\mathrm{p}_a(\hat{B})$. This will not be possible if we demand that $K_X^2<(p-3)/72$, which implies that $\mathrm{p}_a(B)<(p-1)/2$. Suppose then that $\bar{B}$ is smooth. Then since it has a global vector field with fixed points (the preimages of the singular points of $B$), $\bar{B}=\mathbb{P}^1_L$. Now straightforward arguments easily show that $X$ is uniruled. Therefore, from the diagram~\ref{sec2-diagram-1} it follows that $Z$ is a uniruled $K3$ surface. But it is known~\cite[Page 395]{Hu16} that every uniruled $K3$ surface is in fact unirational. Moreover since $\pi_1(Z)=\{1\}$, it follows that $\pi_1(X)=\{1\}$. 
 
 To summarize, it has been shown that in this case, if $K_X^2<(p-3)/72$ then $X$ is unirational and simply connected.

\textit{Case 2.} Suppose that $B$ is not irreducible. Let then $B=n_1B_1+\cdots +n_kB_k$, $k\geq 2$, its decomposition into prime divisors. Then by repeating the arguments given in   Case 1 on the irreducible components $B_i$ of $B$ we conclude again that if $K_X^2<(p-3)/72$, then $X$ is simply connected and unirational.

 \textbf{Case 3.2.} Suppose that $D$ is of additive type, i.e., $D^p=0$. In this case I will show that if $(K_X^2)^3+(K_X^2)^2<p-3$ and $K_X^2<\frac{p-3}{44\cdot 45}$, then $X$ is unirational and simply connected. In particular, this happens if 
 \[
 K_X^2< \mathrm{min}
 \left\{
 \frac{p-3}{44\cdot 45}, \sqrt[3]{\frac{p-3}{2}}
 \right\}.
 \]
 
 The main idea for proving this is the following. I will show that there exists a "small" positive number $\nu$ such that $\dim|\nu K_Y| \geq 1$. Therefore, if $V(0)$ is the eigenspace of 0 of $D_{\nu}^{\ast} \colon H^0(\omega^{\nu}_X)\rightarrow H^0(\omega_X^{\nu})$, then $\dim V(0)\geq 2$ and therefore there exists a one dimensional family of integral curves of $D$ of small genus. Then the result will follow by applying Tate's theorem~\cite{Sch09},~\cite{Ta52} on the general member of $V(0)$, or  its irreducible components if it is reducible.
 
 The main steps of the proof are the following.
 
 Let $F=\sum F_{j=1}^nF_j$ be the reduced $g$-exceptional divisor. Then write $F=F^{\prime}+F^{\prime\prime}$, where $F^{\prime}=\sum_{j=1}^rF_j $, where $F_{j}$, $j=1,\ldots,r$ are the $g$-exceptional curves which are not $h$-exceptional, and $F^{\prime\prime}=\sum_{j=r+1}^nF_j$ are the $g$-exceptional curves that are also $h$-exceptional. Notice that $F^{\prime}\not=0$ because if that was the case then there would be a birational morphism $\psi \colon Y \rightarrow Z$. Then by the adjunction formula, 
 \[
 K_Y=\psi^{\ast}K_Z+ \tilde{F}=\tilde{F},
 \]
since $K_Z=0$, where $\tilde{F}$ is a $\psi$-exceptional divisor. Then $K_Y^2=\tilde{F}^2\leq 0$, which is impossible since $K_Y$ is ample.
 
 Then I will show that at least one of the following is true.
 
  \begin{enumerate}
 \item $\dim H^0(\mathcal{O}_Y(2K_Y))\geq 2$ and hence $\dim|2K_Y|\geq 1$.
 \item There exists a divisor $B=\sum_{j=1}^r n_j F_j$, and a positive number $\nu\leq K_X^2$, such that $\dim |\nu K_{Y^{\prime}}+B| \geq 1$. Moreover, the linear system $|\hat{B}|$, where $\hat{B}=h_{\ast}B$ in $Z$ is either base point free or its moving part is base point free. This implies that $\dim|\nu K_Y| \geq 1$ and moreover, there exists an irreducible component $\tilde{W}$ of the general member of $\nu K_Y|$ such that $\dim|\tilde{W}| \geq 1$. Then if 
 $W=\pi^{\ast} \tilde{W}$, $\tilde{W}$ is reduced and irreducible and $\dim V(0)\geq 2$, where $V(0)$ is the eigenspace of $0$ of the map 
 $D_W^{\ast} \colon H^0(\mathcal{O}_X(W)) \rightarrow H^0(\mathcal{O}_X(W))$, which exists by Proposition~\ref{sec1-prop1}. Then the result will follow by applying Tate's theorem on the general member 
 of $|W|$.
  \end{enumerate}
 
 Suppose that $\dim H^0(\mathcal{O}_Y(2K_Y))\geq 2$.  Then $\dim V(0)\geq 2$, where $V(0)$ is the eigenspace of zero of $D_2^{\ast}\colon H^0(\omega_X^2) \rightarrow H^0(\omega_X^2)$. Then by arguing as in the case when $D^p=D$ and using Tate's theorem we see that if $K_X^2<(p-3)/6$, this case is impossible.
 
 Assume then that $K_X^2<(p-3)/6$ and therefore, as shown above, $H^0(\mathcal{O}_Y(K_Y))=H^0(\mathcal{O}_Y(2K_Y))=k$ and hence $p_g(Y)=1$.

 First I will show that $Y$ has rational singularities. Indeed. The Leray spectral sequence for $g$ gives 
 \[
 0 \rightarrow H^1(\mathcal{O}_{Y}) \rightarrow H^1(\mathcal{O}_{Y^{\prime}})\rightarrow H^0(R^1g_{\ast}\mathcal{O}_{Y^{\prime}}) 
 \rightarrow H^2(\mathcal{O}_Y)\rightarrow H^2(\mathcal{O}_{Y^{\prime}}) \rightarrow H^1(R^1g_{\ast}\mathcal{O}_{Y^{\prime}}).
 \] 
 Now since $g$ is birational it follows that $H^1(R^1g_{\ast}\mathcal{O}_{Y^{\prime}})=0$. Moreover, by Serre duality, $H^2(\mathcal{O}_Y)\cong H^0(\mathcal{O}_Y(K_Y))=k$ and 
 $H^2(\mathcal{O}_{Y^{\prime}})\cong H^0(\mathcal{O}_{Y^{\prime}}(K_{Y^{\prime}}))=k$ and $H^1(\mathcal{O}_{Y^{\prime}})=0$, since $Z$ is a $K3$ surface. Hence from the Leray sequence it follows that $H^1(\mathcal{O}_Y)=0$ and  $ R^1g_{\ast}\mathcal{O}_{Y^{\prime}}=0$. Therefore $Y$ has rational singularities as claimed. In particular, every $g$-exceptional curve is a smooth rational curve.

 Let $Q\in Y$ be an index 1 singular point of $Y$. Let $F_Q$ be the set of $g$-exceptional curves over $Q$. Then I will show that either every $g$-exceptional curve over $Q$ is also $h$-exceptional, or that $F_Q\subset Y^{\prime} - \mathrm{Ex}(h)\cong Z-h(\mathrm{Ex}(h))$, where $\mathrm{Ex}(h)$ is the exceptional set of $h$. 
 
 Indeed. Since $Y$ has rational singularities, $Q\in Y$   must be a DuVal singularity. Therefore, if $F_i$ is a $g$-exceptional curve over $Q$, then $a_i=0$ in the adjunction formula~\ref{sec2-eq-1} and $K_{Y^{\prime}}\cdot F_i=0$. But since $Z$ is a $K3$ surface, $K_{Y{\prime}}=\sum_{j=1}^m b_j E_j$. Therefore, either $F_i$ is $h$-exceptional, or $F_i  \cdot E_j=0$, for all $j=1,\dots, m$ and hence $F_i \cap \mathrm{Ex}(h) =\emptyset$. Let $F^{\prime}_Q \subset F_Q$ be the maximal subset of $F_Q$ which consists of $g$-exceptional curves which are also $h$-exceptional. Suppose that $F^{\prime}_Q$ is empty. Then since for any $F_i \in F_Q$, $F_i \cdot K_{Y^{\prime}}=0$, it follows that $F_i \cdot E_s =0$, for any $h$-exceptional curve $E_s$. Therefore $F_i \cap \mathrm{Ex}(h)=\emptyset$, for all $F_i \in F_Q$ and hence $F_Q \subset Y^{\prime} - \mathrm{Ex}(h)$, as claimed. Suppose that $F^{\prime}_Q\not= \emptyset$ and that $F^{\prime}_Q \underset{\not=}{\subset} F_Q$. Then since the exceptional set of $g$ is connected over a neighborhood of $Q$, there exists a $g$-exceptional curve $F_i \in F_Q -F^{\prime}_Q$ such that $F_i$ intersects at least one curve in $F^{\prime}_Q$. But then since $K_{Y{\prime}}\cdot F_i =0$, it follows that $F_i$ does not meet any $h$-exceptional curve which is impossible because it meets at least one member of  $F^{\prime}_Q$ which by definition is $h$-exceptional. 
 
 Let $\hat{F}_i=h_{\ast}F_i$, $i=1,\dots, r$, be the birational transforms of the $F_i$ in $Z$. Consider next cases with respect to whether the curves $\hat{F}_i$ are either all smooth or there exists a singular one among them. As has been shown above, if a $g$-exceptional curve $F_i$ exists such that $\hat{F}_i=h_{\ast}F_i$ is singular, then $F_i$ must lie over an index $\geq 2$ point of $Y$. Therefore $F_i$,  appears with a  positive coefficient in the adjunction formula~\ref{sec2-eq-1}. Note also that $Y^{\prime}$ has singular points of index $\geq 2$ because otherwise, since it has rational singularities, it would have  at worst DuVal singularities and hence $K_{Y^{\prime}}=g^{\ast}K_Y$. But then $K_Y^2=K_{Y^{\prime}}^2\leq 0$, which is impossible since $K_Y$ is ample.
 
 \textit{Case 1.} Suppose that there exists an $1\leq i \leq r$ such that $\hat{F}_i$ is singular. In this case I will show that 
  \[
 (K_X^2)^3+(K_X^2)^2>p-3.
 \]

From the previous discussion, $F_i$ must lie over a singular point of $Y$ of index $\geq 2$. After a renumbering of the $g$-exceptional curves we can assume that $i=1$. Then by the adjunction formula
 \[
 \hat{F}_1^2=2p_a(\hat{F}_1)-2-K_Z\cdot \hat{F}_1 =2p_a(\hat{F}_1)-2\geq 0.
 \]
 Hence the linear system $|\hat{F}_1|$ in $Z$ is base point free~\cite[Propositions 3.5,  3.10]{Hu16}.
 
\begin{claim}\label{sec3-claim-1}  Let $Q\in \hat{F}_1$ be a singular point of $\hat{F}_1$ and $m=m_Q(\hat{F}_1)$ be the multiplicity of the singularity. Then
\begin{gather}\label{sec3-eq-14}
m_Q(\hat{F}_1) \leq K_X^2.
\end{gather} 
 \end{claim}
In order to prove the claim, observe the following. Over a neighborhood of any singular point of $\hat{F}_1$, $F_1$ can meet at most two distinct $h$-exceptional curves $E_i$ and $E_j$, and moreover it must intersect each one of them with multiplicity 1. Indeed.

Suppose that $F_1$ meets three distinct $h$-exceptional curves $E_{i}$, $E_j$ and $E_s$ (over the same point of $Z$). Since $h$ is a composition of blow ups, it follows that  $E_i\cap E_j \cap E_s=\emptyset$. Hence the intersection of $F_1$ and $E_i\cup E_j \cup E_s$ consists of at least two distinct points, say $P$ and $Q$. Up to a change of indices we can assume hat $P\in E_i$ and $Q\in E_j$.  Then the union $\mathrm{Ex}(h)\cup F_1$, where $\mathrm{Ex}(h)$ is the exceptional set of $h$, contains a cycle. Therefore from the equations~\ref{sec2-eq-1} it follows that
\begin{gather}\label{sec3-eq-14}
K_Y=\sum_{j=1}^n b_j\tilde{E}_j,
\end{gather}
where $\tilde{E}_j=g_{\ast}E_j$, $j=1,\ldots, n$. Moreover if $E_j^2=-1$, then $\tilde{E}_j\not=0$. But then, if $F_1$ meets at least two distinct $h$-exceptional curves, $\cup_{j=1}^n \tilde{E}$ contains either a singular curve or a cycle. In any case,  if $\tilde{C}=\sum_{j=1}^n b_j\tilde{E}_j$ then  
$H^1(\mathcal{O}_{\tilde{C}})\not=0$. But then from the equation in cohomology
\[
H^1(\mathcal{O}_Y) \rightarrow H^1(\mathcal{O}_{\tilde{C}}) \rightarrow H^2(\mathcal{O}_Y(-\tilde{C}))\rightarrow H^2(\mathcal{O}_Y)\rightarrow 0,
\]

and since $H^1(\mathcal{O}_Y)=0$, $H^2(\mathcal{O}_Y)=k$, it follows that $\dim H^2(\mathcal{O}_Y(-\tilde{C}))\geq 2$.  Then by duality,
\[
\dim H^0(\mathcal{O}_Y(K_Y+\tilde{C})=\dim H^0 (\mathcal{O}_Y(2K_Y))\geq 2,
\]
a contradiction. Hence $F_1$ meets at most two distinct $h$-exceptional curves. Suppose that $F_1$ meets an $h$-exceptional curve $E_i$ and $E_i \cdot F_1 \geq 2$. Then there are two possibilities. Either $E_i$ is also $g$-exceptional or it is not. Suppose that $E_i$ is $g$-exceptional. But this is impossible because $Y$ has rational singularities and in such a case two $g$-exceptional curves cannot intersect with multiplicity bigger than one. Suppose that $E_i$ is not $g$-exceptional. Then $\tilde{E}_i=g_{\ast}E_i$ is singular and therefore $h^1(\mathcal{O}_{\tilde{E}_i})\geq 1$. But then $h^1(\mathcal{O}_{\tilde{C}})\geq 1$ and hence arguing as before we see that $\dim H^0(\mathcal{O}_Y(2K_Y))\geq 2$, which is impossible. Hence it has been shown that over a neighborhood of any singular point of $\tilde{F}_1$, $F_1$ meets at most two $h$-exceptional curves with multiplicity at most one. 

Next I will show that 
\begin{gather}\label{sec3-eq-15}
m_Q(\hat{F_i}) \leq K_{Y^{\prime}}\cdot F_i.
\end{gather}
The map $h$ is a composition of blow ups of points of $Z$. Since $\hat{F}_i$ is singular, $h$ must blow up the singular points of $\hat{F}_i$. Let $h_1 \colon Y_1 \rightarrow Z$ be the blow up of $Q\in Z$. Then there exists a factorization
\[
\xymatrix{
Y^{\prime}\ar[dr]^{h_2}\ar[rr]^h & & Z \\
    & Y_1\ar[ur]^{h_1} & 
    }
\]
Then also $h_1^{\ast}\hat{F}_i = (h_1)_{\ast}^{-1}\hat{F}_i +m_Q(\hat{F}_i)E_1$, where $E_1$ is the $h_1$-exceptional curve and $  (h_1)_{\ast}^{-1}\hat{F}_i$ is the birational transform of $\hat{F}_i$ in $Y_1$.     
From this it follows that $E_1\cdot (h_1)_{\ast}^{-1}\hat{F}_i=m_Q(\hat{F}_i)$.  Also $K_{Y_1}=h_1^{\ast}K_Z+E_1=E_1$. Therefore $K_{Y_1}\cdot (h_1)_{\ast}^{-1}\hat{F}_i=m_Q(\hat{F}_i)$. Moreover,  
\[
K_{Y^{\prime}}=h_2^{\ast}K_{Y_1}+E^{\prime},
\]
where $E^{\prime}$ is an effective $h_2$-exceptional divisor. But then
\[
K_{Y^{\prime}}\cdot F_i = h_2^{\ast}K_{Y_1}\cdot F_i+E^{\prime}\cdot F_i \geq K_{Y_1}\cdot  (h_2)_{\ast}F_i =K_{Y_1}\cdot (h_1)_{\ast}^{-1}\hat{F}_i =m_Q(\hat{F}_i).
\]
This proves the claim. 

As it has been shown earlier, $F_i$ meets at most two $h$-exceptional curves $E_j$ and $E_s$, with the possibility $j=s$, each one of them  with intersection multiplicity one. 

Suppose that $E_i \not= E_j$ and that $F_i$ intersects $E_j$ and $E_s$ at the same point $Q$. Hence $E_j\cap E_s \cap F_i \not=\emptyset$. Then since $Y$ has rational singularities it is not possible that $E_j$ and $E_s$ are both $g$-exceptional. 

Suppose that $E_s$ is $g$-exceptional but $E_j$ is not $g$-exceptional. Then $g_{\ast}E_j$ would be singular.  But then from the equation~\ref{sec3-eq-14} and the arguments following it, we get again 
 that $\dim H^0(\mathcal{O}_Y(2K_Y)) \geq 2$, a contradiction.
 
 Hence neither of $E_j$ and $E_s$ is $g$-exceptional. Now write
 \[
 K_{Y^{\prime}}=b_jE_j+b_sE_s+\sum_{r\not= j,s}b_r E_r.
 \]
 Then from the equation~\ref{sec3-eq-15} and the facts that $E_j\cdot F_i=E_s\cdot F_i=1$, $F_i \cdot E_r=0$, for $r\not=j,s$,  it follows that 
 \[
 m_Q(\hat{F}_i)\leq K_{Y^{\prime}}\cdot F_i =b_j+b_s.
 \]
 Then from the equation~\ref{sec3-eq-14} and the fact that $E_j$ and $E_s$ are not $g$-exceptional it follows that 
 \[
 K_Y=b_j\tilde{E}_j+b_s\tilde{E}_s+\tilde{W},
 \]
 where $W$ is an effective divisor. Then since $K_X=\pi^{\ast}K_Y$ we get that 
\[
K_X=b_j\pi^{\ast}\tilde{E}_j+b_s\pi^{\ast}\tilde{E}_s+\pi^{\ast}\tilde{W}.
\]
Now considering that $K_X$ is ample we get that 
\[
m_Q(\hat{F}_i)\leq K_{Y^{\prime}}\cdot F_i =b_j+b_s\leq b_j\pi^{\ast}\tilde{E}_j \cdot K_X+b_s \pi^{\ast}\tilde{E}_s \cdot K_X\leq K_X^2,
\]
as claimed.

Suppose finally that $E_j=E_s$, i.e., $F_i$ meets exactly one $h$-exceptional curve. Then $K_{Y^{\prime}}\cdot F_i =b_j$. If $E_j$ is not $g$-exceptional then the previous argument proves the claim. Suppose that $E_j$ is also $g$-exceptional. Then there exists a $-1$ $h$-exceptional curve $E_{\lambda}$ such that $b_{\lambda}\geq b_j$. The previous argument now shows that $b_{\lambda} \leq K_X^2$ and hence \[
m_Q(\hat{F}_i) \leq b_j\leq b_{\lambda}\leq K_X^2.
\]
This concludes the proof of Claim~\ref{sec3-claim-1}.

\begin{claim}\label{sec3-claim-2} Let $B$ be any member of the linear system $|(K_X^2)K_{Y^{\prime}}+F_i|$.  Then  
\begin{gather}\label{sec3-eq-16}
B\sim W^{\prime}+\sum_{i=1}^m \gamma_i E_i,
\end{gather}
where $\gamma_i\geq 0$ for all $i$ and $W^{\prime}$ is the birational transform in $Y^{\prime}$ of a smooth curve $W$ in $Z$ such that  $|W|$ is base point free and $\mathrm{p}_a(W) \geq 1$.
\end{claim}

By~\cite[Proposition 3.5 and 3.10]{Hu16}, the linear system $|\hat{F}_i|$ is base point free and contains a smooth curve. Let $W \in |\hat{F}_i|$. be a general member. Then $W$ is reduced and irreducible and moreover it does not pass through $h(\mathrm{Ex}(h))$. Let $W^{\prime}$ be the birational transform of $W$ in $Y$. Then $W^{\prime}\cong W$. Now from Proposition~\ref{sec1-prop-6} it follows that 
\begin{gather}\label{sec3-eq-17}
mK_{Y^{\prime}}-h^{\ast}\hat{F}_i+F_i=\sum_{i=1}^m\gamma_i E_i,
\end{gather}
where $\gamma_i \geq 0$, for all $1\leq i \leq m$, and $m$ is the maximum of the multiplicities of the singular points of $\hat{F}_i$. But from Claim~\ref{sec3-claim-1} it follows that $m\leq K_X^2$. Hence 
\begin{gather}\label{sec3-eq-18}
(K_X^2)K_{Y^{\prime}}-h^{\ast}\hat{F}_i+F_i=\sum_{i=1}^m\gamma^{\prime}_i E_i,
\end{gather}
for some $\gamma_i^{\prime}\geq 0$, for all $1\leq i \leq m$. Let now  $W \in |\hat{F}_i|$ be a general member. Then $W^{\prime}=h^{\ast}\hat{F}_i=F_i +(h^{\ast}\hat{F}_i-F_i)$. Then from the equation~\ref{sec3-eq-17} it follows that
\[
(K_X)^2K_{Y^{\prime}}+F_i=(K_X^2)K_{Y^{\prime}}+W^{\prime}-h^{\ast}\hat{F}_i+F_i=W^{\prime}+\sum_{i=1}^m\gamma^{\prime}_i E_i,
\]
for some $\gamma_i^{\prime}\geq 0$, $1\leq i \leq m$. This concludes the proof of Claim~\ref{sec3-claim-2}.

Now pushing down to $Y$ by $g_{\ast}$, and considering that $F_i$ is $g$-exceptional, we see that 
\begin{gather}\label{sec3-eq-19}
(K_X^2)K_Y\sim \tilde{W}+\sum_{j=1}^m\gamma_j^{\prime}\tilde{E}_j.
\end{gather}
 Moreover notice that from the construction of $W$, $\dim|\tilde{W}| \geq 1$. Now pulling up to $X$ by $\pi$ we get that
 \begin{gather}\label{sec3-eq-20}
(K_X^2)K_X\sim \hat{W}+\sum_{j=1}^m\gamma_j^{\prime}\pi^{\ast}\tilde{E}_j,
\end{gather} 
 where $\hat{W}=\pi^{\ast}\tilde{W}$. Moreover, unless $K_X^2>\sqrt{p}$, $\hat{W}$ is reduced and irreducible. Indeed, if $\hat{W}$ was not reduced, then $\pi^{\ast}=p\Delta$~\cite{RS76}, which would imply $K_X^2>\sqrt{p}$. From now on assume that $K_X^2<\sqrt{p}$ and hence $\hat{F}$ is reduced and hence an integral curve of $D$. 
 
 Work now with the linear system $|\hat{W}|$. Let $V(0) \subset |\hat{W}|$ be the eigenspace of $0$ of $D_w^{\ast} \colon H^0(\mathcal{O}_X(\hat{W})) \rightarrow H^0(\mathcal{O}_X(\hat{W}))$. Then it has been shown that $\dim V(0)\geq 2$.  Moreover, every member of $V(0)$ is reduced and irreducible and it contains fixed points of $D$. Indeed. suppose that $C\in V(0)$ was a member that did not contain any fixed points, of $D$. Then if $\tilde{C}=\pi(C)$, $C=\pi^{\ast}(\tilde{C})$ and $\tilde{C}$ is contained in the smooth part of $Y$. Moreover, $\tilde{C} \sim \tilde{W}$. Hence $\tilde{W}$ is Cartier in $Y$. Then from the equation~\ref{sec3-eq-20} we get that 
 \[
 (K_X^2)^2 =\hat{W}\cdot K_X +\sum_{j=1}^m\gamma_j^{\prime}\pi^{\ast}\tilde{E}_j\cdot K_X\geq  K_X \cdot \hat{W}=\pi^{\ast}K_Y\cdot \pi^{\ast}\tilde{W} =p K_Y\cdot \tilde{W} \geq p,
 \]
 since $\tilde{W}$ is Cartier, and hence $K_X^2\geq \sqrt{p}$, a contradiction.
 
 Now from the equation~\ref{sec3-eq-20} it follows that $K_X\cdot \hat{W} \leq (K_X^2)^2$. Then by the Hodge index theorem we get that $\hat{W}^2\leq (K_X^2)^3$. Then from the adjunction formula we get that 
 
 \begin{gather}\label{sec3-eq-21}
 \mathrm{p}_a(\hat{W})\leq \frac{1}{2}((K_X^2)^3+(K_X^2)^2)+1.
 \end{gather}
 
 Moreover, $\mathrm{p}_a(\hat{W} )\geq 1$. Indeed. Since $\hat{W} \rightarrow \tilde{W}$ is purely inseparable of degree $p$, we see that $\pi$ factors through the geometric Frobenius $F\colon \hat{W} \rightarrow \hat{W}^{(p)}$ and therefore there exists a birational map $\tilde{W} \rightarrow \hat{W}^{(p)}$. Hence $\mathrm{p}_a(\hat{W})=\mathrm{p}_a(\hat{W}^{(p)})\leq \mathrm{p}_a(\tilde{W})$. But there is also a birational map $W^{\prime} \rightarrow \tilde{W}$ and hence $\mathrm{p}_a(\tilde{W}) \geq \mathrm{p}_a(W^{\prime})=\mathrm{p}_a(W) \geq 1$.

 Consider now the general element $C$ of $|\hat{W}|$. Suppose that $C$ is smooth. Suppose that $\mathrm{p}_a(C)\geq 2$. Then since $C$ is an integral curve of $D$, this is impossible since smooth curves of genus at lest 2 do not have nontrivial global vector fields. Suppose that $\mathrm{p}_a(C)=1$. Then $C$ is a smooth elliptic curve with a global vector field. But it has been shown that $C$ contains fixed points of $D$. But this is impossible since vector fields on smooth elliptic curves do not have fixed points. Therefore $C$ must be singular. Then from Tate's theorem again it follows that $(p-1)/2 <\mathrm{p}_a(C)$. This implies from the equation~\ref{sec3-eq-21} that 
 \[
 (K_X^2)^3+(K_X^2)^2>p-3,
 \]
 as was to be shown.
 
 \textit{Case 2.} Suppose that $\hat{F}_i$ is smooth for any $i=1,\ldots, r$. In this case I will show that $ K_X^2> (p-3)/506$.

 Since $\hat{F}_i$ is smooth it follows that $\hat{F}_i\cong \mathbb{P}^1$ and that $\hat{F}_i^2=-2$, for all $i=1,\dots,r$. Consider now cases with respect to whether or not every connected subset of the set $\{\hat{F},\ldots,\hat{F}_r\}$ is contractible. 
 
 \textit{Case 2.1.} Suppose that every connected subset of $\{\hat{F}_1,\ldots,\hat{F}_r\}$ is contractible. Let $\phi \colon Z \rightarrow W$ be the contraction. Since $\hat{F}_i^2=-2$, for all $i=1,\ldots, r$, $W$ has Duval singularities. Therefore $K_Z=\phi^{\ast}K_W$. Hence,  since $K_Z=0$, $K_W=0$. Then there exists a factorization
 \[
 \xymatrix{
 Y^{\prime}\ar[rr]^g\ar[dr]^{\phi h} & & Y\ar[ld]^{\psi}\\
        & W &
 }
 \]
 Hence $K_Y=\psi^{\ast}K_W+\tilde{E}=\tilde{E}$, where $\tilde{E}$ is a divisor supported on the $\psi$-exceptional set. But then $K_Y^2=\tilde{E}^2\leq 0$, which is impossible since $K_Y$ is ample.
 
 \textit{Case 2.2.} There exists at least one connected subset of $\{\hat{F}_1,\ldots,\hat{F}_r\}$ which is not contractible. 
 
 \begin{claim} There exists integers $0 \leq \gamma_j \leq 22$, $j=1,\ldots, r$ such that the linear system $|44K_{Y^{\prime}}+\sum_{j=1}^r\gamma_j F_j|$ has dimension at least one. Moreover, let 
 $B\in |44K_{Y^{\prime}}+\sum_{j=1}^r\gamma_j F_j| $ be any member. Then if $K_X^2 <p/44$, 
 \[
 B\sim W^{\prime}+\sum_{i=1}^m \gamma_i E_i,
\]
where $\gamma_i\geq 0$ for all $i$ and $W^{\prime}$ is the birational transform in $Y^{\prime}$ of a reduced and irreducible curve $W$ in $Z$ such that  $|W|$ is base point free and $\mathrm{p}_a(W) \geq 1$.

 \end{claim}\label{sec3-claim-5}
 
 In order to prove the Claim~\ref{sec3-claim-5} it is necessary to prove first the following.
 
 \begin{claim}\label{sec3-claim-6} 
There exist numbers $0 \leq \gamma_i \leq 22$, $i=1,\dots, r$ such that if $B=\sum_{i=1}^r\gamma_i \hat{F}_i$, then $B\cdot \hat{F}_i \geq 0$, for all $1\leq i \leq r$, and $B^2\geq 0$.
\end{claim}

I  proceed to prove the claims. Let $\{\hat{F}_1,\ldots, \hat{F}_s\}$, $s<r$, be the maximal connected subset of $\{\hat{F}_1,\ldots,\hat{F}_r\}$ which is contractible. Since the rank of $\mathrm{Pic}(Z)$ is at 
most 22~\cite{Hu16} it follows that $r\leq 22$. However, if the Picard number of $Z$ is 22 then $Z$ is in fact unirational~\cite{Li15}. Therefore we can assume that the Picard number of $Z$ is at most 21 and hence $r \leq 21$.

 Let $\phi \colon Z \rightarrow Z^{\prime}$ be the contraction of $\{\hat{F}_1,\ldots, \hat{F}_s\}$. Then $Z^{\prime}$ has DuVal singularities. Since $\cup\hat{F}_{i=1}^r$ is connected, there exists a curve $\hat{F}_j \in \{\hat{F}_{s+1},\ldots, \hat{F}_r\}$,  such that $\hat{F}_j\cap (\cup\hat{F}_{i=1}^s)\not=\emptyset$ and of course $\hat{F}_j$ does not contract by 
$\phi$. Let $F_j^{\prime}=\phi_{\ast}\hat{F}_j$. Observe now that one of the following happens.
\begin{enumerate}
\item $F^{\prime}_j$ is singular. In this case one of the following happens.
\begin{enumerate}
\item $\hat{F}_j$  meets two distinct $\phi$-exceptional curves, say $\hat{F}_{\lambda}$, $\hat{F}_{\mu}$, $1\leq \lambda < \mu \leq s$.  
\item $\hat{F}_j$ meets one $\phi$-exceptional curve $\hat{F}_i$, $i \leq s$, such that $\hat{F}_j \cdot \hat{F}_i \geq 2$. 
\item $\hat{F}_j$ meets exactly one $\phi$-exceptional curve $\hat{F}_i$ and $\hat{F}_i \cdot \hat{F}_j =1$.
\end{enumerate}
\item $F^{\prime}_j$ is smooth.
\end{enumerate}
 
 Suppose that the case 1.a happens. Then let $B=\hat{F}_j +\sum_{i=\lambda}^{\mu} \hat{F}_i$. Then this is a cycle of $-2$ rational curves and $B\cdot \hat{F}_i =0$, for all $i\in \{j, \lambda, \lambda +1,\ldots, \mu\}$, and $B^2=0$.
 
 Suppose that the case 1.b happens. Then let $B=\hat{F}_j+\hat{F}_i$. Then $B\cdot \hat{F}_j \geq 0$, $B\cdot \hat{F}_i \geq 0$ and $B^2\geq 0$.
 
 Suppose that the case 1.c happens. This can happen only when the fundamental cycle of the singularity of $W$ is not reduced, i.e., when $W$ has either a $D_s$, $E_6$, $E_7$ or $E_8$ singularity. 
 
 Suppose that $W$ has a $D_s$ singularity.  The fundamental cycle of the singularity is $\hat{F}_1 +2\sum_{i=1}^{s-2}\hat{F}_i +\hat{F}_{s-1}+\hat{F}_s$. Hence in this case  $\hat{F}_j$ must intersect some $\hat{F}_i$, $2\leq i \leq s-2$. Let $B=\hat{F}_j+\hat{F}_{i-1}+2\sum_{k=1}^{s-2}\hat{F}_k +\hat{F}_{s-1}+\hat{F}_s$. Then  $B\cdot \hat{F}_j = 0$, $B\cdot \hat{F}_k = 0$, $i-1\leq k \leq s$ and $B^2=0$. 
 
 The cases when $W$ has $E_6$, $E_7$ or $E_8$ singularities are treated similarly. 
 
 Suppose finally that case 2 happens, i.e., $F^{\prime}_j$ is smooth. Then write
 \[
 \phi^{\ast}F_j^{\prime}=\hat{F}_j+\sum_{i=1}^sa_i\hat{F}_i.
 \]
 Let $m$ be the index of $F^{\prime}_j$ in $S$. Then according to Proposition~\ref{sec1-prop-7}, $m\in\{2,3,4, s+1\}$ (the exact value of $m$ depends on the type of singularities of $S$). Moreover, if $S$ has an $A_s$ or $D_s$ singularity, then $ma_i \leq s$, for all $i=1,\ldots, s$. If $S$ has an $E_6$ singularity then $ma_i \leq 6$ for all $i$ and if $S$ has an $E_7$ singularity then $ma_i \leq 7$ for all $i$. In any case $ma_i$ are positive integers at most 21, for all $i=1,\ldots, s$, and $m\leq s+1\leq 22$. Let $\gamma_i= ma_i$, for all $i=1,\ldots, s$ and $\gamma_j=m$. Let also
 \[
 B=m\phi^{\ast} F_j^{\prime}=\gamma_j\hat{F}_j+\sum_{i=1}^s\gamma_i\hat{F}_i.
 \]
 Then $B\cdot \hat{F}_i =0$, $i=1,\dots, s$, and $B\cdot \hat{F}_j=m(F^{\prime}_j)^2\geq 0$ (if $(F^{\prime})^2<0$, then the set $\{\hat{F}_j, \hat{F}_1, \ldots, \hat{F}_s\}$ would be contractible which is not true). Moreover, $B^2\geq 0$. This concludes the proof of Claim~\ref{sec3-claim-6}.
 
 So it has been proved that there exists a nontrivial effective divisor $B=\sum_{i=1}^r \gamma_i \hat{F}_i$ in $Z$, such that $0\leq \gamma_i \leq 22$, $i=1,\dots,r$, and $B\cdot \hat{F}_i \geq 0$ for all $i=1,\dots, r$ and $B^2\geq 0$. In particular,  if three of the $\hat{F}_i$ meet at a common point or two have a tangency then $B$ is reduced.  Now since $\hat{F}_i$ is smooth for all $i$, every multiple $\gamma_i \hat{F}_i$ can be considered singular with multiplicity $\gamma_i \leq 22$ at every point. If two, say $\hat{F}_i$ and $\hat{F}_j$ meet at a point with multiplicity 1 then $B$ has at this point multiplicity $\gamma_i+\gamma_j\leq 22+22=44$. Therefore from Proposition~\ref{sec1-prop-6} it follows that $44K_{Y^{\prime}}-h^{\ast}B+B^{\prime}$ is an effective divisor, where $B^{\prime}=\sum_{i=1}^r \gamma_i F_i$. 
 
 Consider now cases with respect to $B^2$.
 
 Suppose that $B^2=0$. Then by~\cite[Proposition 3.10]{Hu16}, the linear system $|B|$ is base point free. Moreover, by~\cite[Page 31]{Hu16},~\cite{Za44},~\cite[Theorem 6.3]{Jou83}, if $p\not=2,3$, 
 $B\sim p^{\nu}W$, where $W$ is a smooth irreducible elliptic curve. In fact $|W|$ is also base point free~\cite[Proposition 3.10]{Hu16}. I claim that if $\nu >0$, then $K_X^2>p^{\nu}/44$. Indeed. 
 \[
 44K_{Y^{\prime}}+B^{\prime}=44K_{Y^{\prime}}+B^{\prime}-h^{\ast}B+h^{\ast}B=(44K_{Y^{\prime}}-h^{\ast}B+B^{\prime})+p^{\nu}W^{\prime}=\\
 p^{\nu}W^{\prime} +E,
 \]
 where $E$ is an effective divisor whose prime components are $g$-exceptional and $h$-exceptional curves and $W^{\prime}$ is the birational transform of $W$ in $Y^{\prime}$ ($W$ can be chosen to avoid the points blown up by $h$). Then by pushing down to $Y$ and then pulling up on $X$ we find that
 \begin{gather}\label{sec3-eq-22}
 44K_X=p^{\nu}\pi^{\ast}\tilde{W}+\pi^{\ast}\tilde{E},
 \end{gather}
 where $\tilde{W}=g_{\ast}W$ and $\tilde{E}=g_{\ast}E$. Also notice that since $W$ moves in $Z$, $W^{\prime}$ is not $g$-exceptional and hence $\tilde{W}\not=0$. Then, since $K_X$ is ample, it follows that $44K_X^2\geq p^{\nu}$. Assume from now on that $K_X^2<p/44$. Then  $\nu=0$. Also note that if $K_X^2<p/44$, then it follows from the equation~\ref{sec3-eq-22} that  $\hat{W}=\pi^{\ast}\tilde{W}$ is irreducible and hence $\hat{W}$ is an integral curve of $D$. Moreover, by the choice of $W$, $\dim |\tilde{W}| \geq 1$. Hence $\dim V(0)\geq 2$, where $V(0)$ is the eigenspace of 0 of the map 
 $D^{\ast}_{\hat{W}}\colon H^0(\mathcal{O}_X(\hat{W}))\rightarrow H^0(\mathcal{O}_X(\hat{W}))$ of Proposirion~\ref{sec1-prop1}. 
 
 Work now with $|\hat{W}|$. From the equation~\ref{sec3-eq-22} and by using the Hodge index theorem as in the previous cases, we find that 
 \[
 \mathrm{p}_a(\hat{W})\leq 22\cdot 45 K_X^2+1.
 \]
 Also notice that the normalization of $\hat{W}$ is isomorphic to $W$ and hence it is a smooth elliptic curve. Repeating now the argument following the equation~\ref{sec3-eq-21} we get that if this is the case then
 $(p-1)/2$ is smaller than $\mathrm{p}_a(\hat{W})$ and hence
 \begin{gather}\label{sec3-eq-24}
 K_X^2>\frac{p-3}{44\cdot 45}.
 \end{gather}
 
 Suppose finally that $B^2>0$. Then $B$ is nef and big. Then by~\cite[Corollary 3.15]{Hu16}, $B\sim mW+C$, where $W$ is a smooth elliptic curve and $C\cong \mathbb{P}^1$. Moreover, as before, the linear system $|W|$ is base point free~\cite[Proposition 3.10]{Hu16}. Repeating now the arguments of the previous case we find that
 \[
 44K_{Y^{\prime}}+B^{\prime}=mW^{\prime}+C^{\prime}+E,
 \]
 where $W^{\prime}$ and $C^{\prime}$ are the birational transforms of $W$ and $C$ in $Y^{\prime}$ and $E$ is effective. Let again $\hat{W}=\pi^{\ast}g_{\ast}W^{\prime}$. Repeating now word by word the arguments of the case when $B^2=0$ for the linear system $|\hat{W}|$ it follows that again in this case the relation~\ref{sec3-eq-24} must hold. Putting together all the previous results concludes the proof of the statement of Case 3.2.
 
 \begin{remark} I believe that by looking at the proof in the case when $D^p=0$ it should be possible to improve the estimate for $K_X^2$. Especially the  denominator $44\cdot 45$ in the inequality $K_X^2<\frac{p-3}{44\cdot 45}$ should be reduced if one studies more carefully the behavior of the coefficients  of the divisor $B=\sum_{j=1}^r\gamma_j\hat{F}_j$ that appears in Claim~\ref{sec3-claim-6}. The bound $\gamma_j \leq 22$ is simply a consequence that $Z$ has Picard number at most 22 which implies that the contraction $Z\rightarrow W$ of the maximal contractible subset of $\{\hat{F}_1,\dots,\hat{F}_r\}$ has DuVal singularities of type $A_s$, $D_s$, $E_6$, $E_7$ or $E_8$, $s\leq 22$. However if one studies carefully the possible contractible configurations of the curves $\hat{F}_i$, many cases about the singularities of $W$ should be excluded and the bound for $K_X^2$ should be improved.  
 \end{remark}

 \textbf{Case 4.} Suppose that $Z$ is Enriques. Then in this case I will show that
 \begin{enumerate}
 \item Suppose that $D^p=D$. Then if $K_X^2<(p-3)/144$, $X$ is unirational and $\pi_1(X)=\mathbb{Z}/2\mathbb{Z}$. 
 \item Suppose that $D^p=0$.  Then if $8(K_X^2)^3+4(K_X^2)^2<p-3$ and $K_X^2<\frac{p-3}{2\cdot 44\cdot 45}$, then $X$ is unirational and $\pi_1(X)=\mathbb{Z}/2\mathbb{Z}$. In particular, this happens if 
 \[
 K_X^2< \mathrm{min}
 \left\{
 \frac{p-3}{2\cdot 44\cdot 45}, \frac{1}{2}\sqrt[3]{\frac{p-3}{2}}
 \right\}.
 \]
  \end{enumerate}
 
 In this case, since we assume $p\not= 2$, $\pi_1(X)=\pi_1(Z)=\mathbb{Z}/2\mathbb{Z}$. Then there exists an \'etale double cover  
 $\nu \colon W\rightarrow X$ of $X$ (we assume that $p\not= 2$). Then $K_W=\nu^{\ast}K_X$ and $K_W^2=2K_X^2$. Also $D$ lifts to a nontrivial global vector field $D^{\prime}$ on $W$. Then in the corresponding diagram~\ref{sec2-diagram-1} for $W$, $Z$ is going to be a $K3$ surface. Then the results from the previous cases for $W$ show the claimed result.

 \subsection{Suppose that $\kappa(Z)=-1$} In this case I will show that $X$ is unirational and that $\pi_1(X)=\{1\}$.
 
 Since $\kappa(Z)=-1$, $Z$ is a ruled surface. Hence there exists a fibration of smooth rational curves $\phi \colon Z \rightarrow B$, where $B$ is a smooth curve. 
 
 Suppose that a $g$-exceptional curve $F$ does map to a point in $B$ by the map $\phi h$. Then there exists a dominant morphism $F\rightarrow B$. But since $B$ is a rational curve then $B\cong \mathbb{P}^1$. Hence $Z$ is rational. Therefore $X$ is unirational and moreover $\pi_1(X)=\pi_1(Y)=\pi_1(Z)=\{1\}$.
 
 Suppose that every $g$-exceptional curve is contracted to a point in $B$ by $\phi h$. Then there exists a factorization
 \[
 \xymatrix{
 Y^{\prime}\ar[r]^g\ar[d]^h & Y \ar[d]^{\psi}\\
 Z \ar[r]^{\phi} & B
 }
 \]
 Let $Y_b$ be a general fiber of $\psi$. Then $Y_b\cong \mathbb{P}^1$. Therefore since $Y_b^2=0$, it follows that $K_Y \cdot Y_b=-2$. But then
 \[
 K_X \cdot \pi^{\ast} Y_b=\pi^{\ast}K_Y \cdot \pi^{\ast}Y_b=pK_Y \cdot Y_b=-2p<0,
 \]
 which is impossible since $K_X$ is ample. Hence this case is impossible. Therefore $B=\cong \mathbb{P}^1$ and hence $X$ is unirational and $\pi_1(X)=\{1\}$, as claimed.
 
 The statement of Theorem~\ref{sec3-th-1} follows by putting together the results that were obtained in every case with respect to the Kodaira dimension $\kappa(Z)$ of $Z$.

 \end{proof}

\section{Vector fields with nontrivial divisorial part}\label{sec-4}
Fix the notation as in Section~\ref{sec-2}. The main result of this section is the following.
\begin{theorem}\label{sec4-theorem-1}
Let $X$ be a smooth canonically polarized surface defined over an algebraically closed field of characteristic $p>0$. Suppose that $X$ admits a nontrivial global vector field $D$ such that $D^p=0$ or $D^p=D$. Assume moreover that the divisorial part of $D$ is not trivial. Then if $K_X^2<(p-3)/156$, $X$ is unirational and $\pi_1(X)=\{1\}$.
\end{theorem}

According to~\cite[Theorem 6.1 and its proof]{Tz17}, if $K_X^2<p$ and the divisorial part $\Delta$ of $D$ is not zero, then $\kappa(Z)=-1$. Assume from now on then that $K_X^2<p$. Then only the case $\kappa(Z)=-1$ needs to be studied.

Since $\kappa(Z)=-1$, $Z$ is a ruled surface. Therefore there exists a fibration of smooth rational curves $\phi \colon Z \rightarrow B$, where $B$ is a smooth curve. In order to prove the theorem it suffices to show that $B\cong \mathbb{P}^1$. 

Suppose that there exists a $g$-exceptional curve which dominates $B$. Then since every $g$-exceptional curve is rational, $B\cong \mathbb{P}^1$. 

Suppose then that every $g$-exceptional curve is contracted to a point in $B$ by $\phi h$.  Then there exists a factorization

\begin{gather}\label{sec4-diag-1}
\xymatrix{
Y^{\prime}\ar[r]^g\ar[d]^h & Y \ar[d]^{\psi}\\
Z\ar[r]^{\phi} & B
}
\end{gather}
where the general fiber of $\psi$ is a smooth rational curve. Let $\sigma \colon X \rightarrow B$ be the composition $\psi\pi$. Finally notice that since the $g$-exceptional set is contained in fibers of $\phi h$, $Y$ has rational singularities.

In order to show that $B\cong \mathbb{P}^1$ I will show that there exists a rational curve (in general singular) $C$ in $X$ which dominates $B$. The method to find such a rational curve is to show that there exists an integral curve $C$ of $D$ on $X$ which dominates $B$. Then by Corollary~\ref{sec1-cor2}, if the arithmetic genus of $C$ is small compared to the characteristic $p$, $C$ is rational. Finally, integral curves of $D$ will be found by utilizing Proposition~\ref{sec1-prop1}.

By~\cite[Theorem 1.20]{Ek88}, the linear system $|3K_X|$ is base point free. Then by~\cite[Theorem 6.3]{Jou83},~\cite{Za44}, the general member of $|3K_X|$ is of the form $p^{\nu} C$, where $C$ is an irreducible and reduced curve. Suppose that $\nu >0$. Then $K_X^2>p/3$. Assume then from now on that $K_X^2<p/3$. Then the general member of $|3K_X|$ is reduced and irreducible (but perhaps singular). 

Consider  cases with respect to whether $|3K_X|$ contains an integral curve of $D$.

\textbf{Case 1.} Suppose there exists an irreducible and reduced curve $C\in|3K_X|$ which is an integral curve of $D$. In this case I will show that if $K_X^2<(p-3)/12$, then $B\cong \mathbb{P}^1$.

Since $C$ is an integral curve of $D$, $D(I_C) \subset I_C$, where $I_C$ is the ideal sheaf of $C$ in $X$ and hence $D$ induces a vector field on $C$. However it is possible that $D(\mathcal{O}_X)\subset I_C$ and hence the restriction of $D$ on $C$ is zero. Suppose that this is the case. Then $C$ is contained in the divisorial part $\Delta$ of $D$. Then $\Delta=C+\Delta^{\prime}=3K_X+\Delta^{\prime}$, where $\Delta^{\prime}$ is an effective divisor. Then from the adjunction formula~\ref{sec2-eq-2} for $\pi$, we get that 
\begin{gather}\label{sec4-eq-1}
(4-3p)K_X=\pi^{\ast}K_Y+(p-1)\Delta^{\prime}.
\end{gather}
Let $b\in B$ be a general point, $Y_b =\psi^{-1}(b)$ and $X_b=\pi^{\ast}Y_b$. Then $Y_b\cong \mathbb{P}^1$. Moreover, since $Y_b^2=0$, it follows from the adjunction formula that $K_Y \cdot Y_b=-2$. Then from the previous equation we get that
\begin{gather}\label{sec4-eq-2}
(4-3p)K_X\cdot X_b=\pi^{\ast}K_Y\cdot \pi^{\ast}Y_b +(p-1)\Delta^{\prime} \cdot X_b=pK_Y \cdot Y_b+(p-1)\Delta^{\prime}\cdot X_b=\\
-2p+(p-1)\Delta^{\prime}\cdot X_b.\nonumber
\end{gather}
Since $K_X$ is ample, the left hand side of the previous equation is negative while the right hand side is positive if $\Delta^{\prime}\cdot X_b \geq 3$. Hence $\Delta^{\prime}\cdot X_b \leq 2$. 

Suppose that $\Delta^{\prime}\cdot X_b=0$. Then $(4-3p)K_X \cdot X_b =-2p$. The only solutions to this is $K_X\cdot X_b=1$ and $p=4$, and $K_X\cdot X_b=2$ and $p=1$. Both solutions are not possible 
since $p$ is prime. Similar considerations show that the cases $K_X\cdot X_b=1$ and $K_X\cdot X_b=2$ are also impossible. 

Therefore it has been shown that the restriction of $D$ on $C$ is not zero. Now since $C\in |3K_X|$ it follows from the adjunction formula that $\mathrm{p}_a(C)=6K_X^2+1\geq 7$. Suppose that $C$ is smooth. Then the restriction of $D$ on $C$ is zero since smooth curves of genus $\geq 2$ do not have nontrivial global vector fields. Hence $C$ is singular. Then from Proposition~\ref{sec1-prop4} and Corollary~\ref{sec1-cor2},  if $\mathrm{p}_a(C)<(p-1)/2$ then the restriction of $D$ on $C$ fixes its singular points and lifts to the normalization $\bar{C}$ of $C$. Moreover, $\bar{C}\cong \mathbb{P}^1$. In particular this happens if $6K_X^2+1<(p-1)/2$, or equivalently if $K_X^2<(p-3)/12$. Suppose that this is the case. Then $C$ is a rational curve. Also since $C\cdot X_b >0$, for any $b\in B$, $C$ dominates $B$. Hence $B\cong \mathbb{P}^1$.

\textbf{Case 2.} Suppose that $|3K_X|$ does not contain any integral curves of $D$. In this case I will show that if $K_X^2<(p-3)/156$, then $X$ is unirational and $\pi_1(X)=\{1\}$.

\begin{claim} \label{sec4-claim-4}
$K_X \cdot \Delta \leq 3K_X^2$ and $\Delta^2 \leq 9K_X^2$.
\end{claim}

I proceed to show the claim. Let $C \in |3K_X|$ be a general element. Then $C$ is reduced and irreducible and is not an integral curve of $D$. Let $\tilde{C}=\pi(C)$. 
Then $\pi_{\ast}C=\tilde{C}$, $\pi^{\ast}\tilde{C}=pC$~\cite{RS76} and the map $\pi \colon C \rightarrow \tilde{C}$ is birational. Moreover, since $C$ is general and $|3K_X|$ is base point free, $C$ does not contain any isolated singular points of $D$ and hence $\tilde{C}$ is contained in the smooth locus of $Y$. Then 
from the adjunction formula for $\tilde{C}$ in $Y$ we get that
\begin{gather}\label{sec4-eq-3}
2\mathrm{p}_a(\tilde{C})-2=K_Y \cdot \tilde{C}+\tilde{C}^2=\pi^{\ast}K_Y \cdot C +pC^2=K_X\cdot C-(p-1)\Delta\cdot C +pC^2 =\\
3K_X^2-3(p-1)K_X\cdot \Delta +9pK_X^2=(9p+3)K_X^2-3(p-1)K_X \cdot \Delta.\nonumber
\end{gather}
Considering now that $2\mathrm{p}_a(C)-2 =12K_X^2$ we get that
\begin{gather}\label{sec4-eq-5}
2(\mathrm{p}_a(\tilde{C})-\mathrm{p}_a(C))=3(p-1)( 3K_X^2-K_X \cdot \Delta).
\end{gather}
But since the map $C \rightarrow \tilde{C}$ is birational, it follows that $\mathrm{p}_a(\tilde{C})\geq \mathrm{p}_a(C)$. Therefore from the equation~\ref{sec4-eq-5} it follows that $K_X\cdot \Delta \leq 3K_X^2$. Then from the Hodge index theorem we get that $\Delta^2 \leq (K_X\cdot \Delta)/K_X^2 \leq 9K_X^2$, as claimed.

Next notice that if $L$ is any line bundle on $X$, then $L^p=\pi^{\ast}M$, where $M$ is a line bundle on $Y$. Then in view of this observation, the adjunction formula~\ref{sec2-eq-2} for $\pi$ becomes
\begin{gather}\label{sec4-eq-6}
K_X+\Delta=\pi^{\ast}K_Y+p\Delta=\pi^{\ast}(K_Y+\tilde{\Delta}),
\end{gather}
where $\tilde{\Delta} \subset Y$ is an effective divisor such that $\pi^{\ast}\tilde{\Delta}=p\Delta$. Hence by Proposition~\ref{sec1-prop1}, there exists a $k$-linear map
\[
D^{\ast} \colon H^0(\mathcal{O}_X(3K_X+3\Delta)) \rightarrow H^0(\mathcal{O}_X(3K_X+3\Delta)).
\]

Let $C \in |3K_X+3\Delta|$ be a curve which corresponds to an eigenvector of $D^{\ast}$. Then by Proposition~\ref{sec1-prop1}, $C$ is an integral curve of $D$. Moreover, by using Claim~\ref{sec4-claim-4} it follows that 
\begin{gather}\label{sec4-eq-7}
K_X\cdot C=3K_X^2+K_X\cdot \Delta \leq 12K_X^2,\\
C^2 =9K_X^2+9\Delta^2+18 K_X \cdot \Delta\leq 144K_X^2 \nonumber
\end{gather}
and therefore from the adjunction formula, $\mathrm{p}_a(C)\leq 78K_X^2+1$. Let now $C=\sum_{i=1}^s n_i C_i$ be the decomposition of $C$ into its prime divisors. Then from the equation~\ref{sec4-eq-7} it follows that $K_X \cdot C_i \leq 12K_X^2$ and from the Hodge index theorem that $C_i^2\leq 144K_X^2$, for all $i=1,\ldots, s$. Therefore $\mathrm{p}_a(C_i)\leq 78K_X^2+1$, for all $i=1,\ldots, s$. 
Then by Corollary~\ref{sec1-cor1}, 
if $K_X^2<p/12$, $D$ induces vector fields on $C_i$, for all $i=1,\ldots, s$. Moreover, if $K_X^2<(p-3)/156$, then $\mathrm{p}_a(C) < (p-1)/2$ and hence from Proposition~\ref{sec1-prop4} it follows that $D$ fixes every singular point of $C_i$, for all $i=1,\ldots, s$. 

From now on assume that $K_X^2<(p-3)/156$ and hence $D$ induces vector fields on each $C_i$ and moreover it fixes its singular points and hence it lifts to the normalization $\bar{C}_i$ of $C_i$. 

Next I will show that for every $i=1,\ldots, s$, there exists a fixed point of $D$ on $C_i$. Suppose that this was not the case and that there exists some $1\leq i \leq s$ such that $D$ has no fixed points on $C_i$. Then if $\tilde{C}_i=\pi(C_i)$, $\tilde{C}_i$ is in the smooth locus of $Y$. Since there are no fixed points of $D$ on $C$, $C \cdot \Delta =0$. Then from the adjunction formula for $\pi$ we get that 
\[
K_X\cdot C_i=\pi^{\ast}K_Y\cdot C_i=\pi^{\ast}K_Y\cdot \pi^{\ast}\tilde{C_i}=pK_Y\cdot \tilde{C_i},
\]
and therefore $K_X\cdot C_i \geq p$. On the otherhand it has been shown that $K_X\cdot C_i \leq 12 K_X^2$. But we are assuming that $K_X^2<(p-3)/156$. Hence $K_X\cdot C_i \leq 12K_X^2\leq p/13$. But this implies that $p<p/13$, which is impossible. Hence there exists fixed points of $D$ on every $C_i$. Hence if the restriction on $D$ on $C_i$ is not zero (equivalently if $C_i$ is not a component of $\Delta$), then by Corollary~\ref{sec1-cor2}, $\bar{C}_i\cong \mathbb{P}^1$.

Now let $\Delta^{\prime}=\sum_{i=1}^{\nu}n_i C_i$, where $C_i$, $1\leq i \leq \nu \leq s$ are the irreducible components of $C$ that are also components of $\Delta$ (and hence the restriction of $D$ on $C_i$ is zero). Let also $Z=\sum_{j=\nu +1}^{s}n_i C_i$, where $C_j$ are the irreducible components of $C$ which are not contained in $\Delta$ and therefore the restriction of $D$ on $C_j$, $j \geq \nu +1$, is not zero (if $\nu =s$ then $Z=0$). Then $C=\Delta^{\prime} +Z$. 

Next I will show that if $K_X^2<p/8$ (which holds under the assumptions $K_X^2<(p-3)/156$), then $Z\not= 0$ and that there is a component of it which dominates $B$. Hence $B$ is rational. 

Suppose that this is not true and that either $Z=0$ or no component of $Z$ dominates $B$. Therefore either $Z=0$ or $Z$ is contained in a finite union of fibers of $\psi h\colon X\rightarrow B$. Let $F$ be a general fiber of $\psi h$. Then in both cases $F \cdot Z=0$. Then if we write $3K_X=C-3\Delta=\Delta^{\prime}+Z-3\Delta$, the adjunction formula for $\pi$ becomes
\[
\Delta^{\prime}+Z=3\pi^{\ast}K_Y +3p\Delta.
\]
Intersecting this with a general fiber $F$ and taking into consideration that $F\cdot Z=0$ and that $F\cdot \pi^{\ast}K_Y=-2p$ we find that
\begin{gather}\label{sec4-eq-8}
\Delta^{\prime}\cdot F=-6p+3p( \Delta \cdot F).
\end{gather}
Now 
\begin{gather}\label{sec4-eq-9}
\Delta^{\prime}\cdot F=\sum_{i=1}^{\nu} n_i( C_i \cdot F)\leq m \left(\sum_{i=1}^{\nu}(C_i \cdot F)\right) \leq m \Delta \cdot F,
\end{gather}
where $m$ is the maximum among the $n_1, \ldots ,n_{\nu}$ such that $C_i \cdot F \not= 0$. Notice that it is not possible that $C_i \cdot F=0$, for all $i=1,\ldots, \nu$. If this was the case, then $\Delta^{\prime}\cdot F=0$. But since also we assume that $Z\cdot F=0$, it would follow that $C\cdot F=0$ and hence $(K_X+\Delta)\cdot F=0$. But then
\[
K_X \cdot F=-\Delta \cdot F \leq 0,
\]
for a general fiber $F$. But this is impossible since $K_X$ is ample. Hence $\Delta^{\prime}\cdot F>0$ and hence $m>0$.

Next I will show that $m \leq 12K_X^2$. Indeed. From the definition of $\Delta^{\prime}$ and the equation~\ref{sec4-eq-7} it follows that
\[
m \leq \sum_{i=1}^{\nu}n_i\leq \sum_{i=1}^{\nu} n_i (K_X \cdot C_i) =K_X \cdot \Delta^{\prime} \leq K_X \cdot C \leq 12K_X^2,
\]
as claimed. Then from the equations~\ref{sec4-eq-8},~\ref{sec4-eq-9} it follows that
\begin{gather}\label{sec4-eq-10}
(12K_X^2-3p)\Delta\cdot F +6p >0.
\end{gather}
Notice now that from the adjunction formula for $\pi$ it follows that
\[
K_X\cdot F=\pi^{\ast}K_Y\cdot F+(p-1)\Delta \cdot F=-2p+(p-1)\Delta\cdot F.
\]
Then since $K_X \cdot F>0$, it follows that $\Delta \cdot F \geq 3$. Now recall that we are assuming that $K_X^2<(p-3)/156$. In particular, $K_X^2<p/12$. Then it is easy to see that
\[
(3p-12K_X^2)\Delta \cdot F-6p >0,
\]
which is a contradiction to the equation~\ref{sec4-eq-10}. Therefore it is not possible that $Z\cdot F=0$. Hence there exists a component $C_i$ of $C$ such the restriction of $D$ on $C$ is not zero and $C_i$ dominates $B$. Then since the normalization of $C_i$ is $\mathbb{P}^1$, it follows that $B\cong \mathbb{P}^1$. This concludes the proof of Theorem~\ref{sec4-theorem-1}.


\end{document}